\documentclass[10pt,a4paper]{amsart}

\usepackage{color}
\usepackage{graphics,epsfig}
\usepackage{amsfonts,amssymb,amsmath,amsthm,times}

\newtheorem{lem}{Lemma}[section]

\newtheorem{thm}[lem]{Theorem}

\theoremstyle{remark}
\newtheorem{rem}[lem]{Remark}

\numberwithin{equation}{section}

\def \R {\mathbb{R}}
\def \N {\mathbb{N}}

\newcommand{\ep}{\varepsilon}
\newcommand{\ue}{u_\ep}
\newcommand{\gep}{g_\ep}
\newcommand{\uk}{u_{\ep,k}}
\newcommand{\lame}{\lambda_\ep}
\newcommand{\lamk}{\lambda_{\ep,k}}

\newcommand{\om}{\Omega}

\newcommand{\edeux}{\displaystyle{\frac{1}{\ep^2}}}
\newcommand{\eun}{\displaystyle{\frac{1}{\ep}}}

\title[Convergence of a mass conserving Allen-Cahn equation]
{Convergence of a mass conserving Allen-Cahn equation whose
Lagrange multiplier is  nonlocal {\it and} local }
\author[M. Alfaro]{Matthieu Alfaro}
\address{Univ. Montpellier 2\\ I3M\\ UMR CNRS 5149\\ CC051\\
 Place Eug\`ene Bataillon\\ 34095 Montpellier Cedex 5\\ France}
\email{malfaro@math.univ-montp2.fr}
\author[P. Alifrangis]{Pierre Alifrangis}
\address{Univ. Montpellier 2\\ I3M\\ UMR CNRS 5149\\ CC051\\
 Place Eug\`ene Bataillon\\ 34095 Montpellier Cedex 5\\ France}
\email{alifrang@math.univ-montp2.fr}
\thanks{The authors are supported
by the French Agence Nationale de la Recherche within the project
IDEE (ANR-2010-0112-01).}
\begin{document}

%\tableofcontents

\begin{abstract}
We consider the mass conserving Allen-Cahn equation proposed in
\cite{Bra-Bre}: the Lagrange multiplier which ensures the
conservation of the mass contains not only nonlocal but also local
effects (in contrast with \cite{Che-Hil-Log}). As a parameter
related to the thickness of a diffuse internal layer tends to
zero, we perform formal asymptotic expansions of the solutions.
Then, equipped with these approximate solutions, we rigorously
prove the convergence to the volume preserving mean curvature
flow, under the assumption that classical solutions of the latter
exist. This requires a precise analysis of the error between the
actual and the approximate Lagrange multipliers.
\end{abstract}

\maketitle

\section{Introduction}\label{s:intro}

\noindent {\bf Setting of the problem.} In this paper, we consider
$\ue=\ue(x,t)$ the solutions of an Allen-Cahn equation with
conservation of the mass proposed in \cite{Bra-Bre}, namely
\begin{equation}\label{mass-modified}
\partial _t \ue=\Delta \ue
+\edeux\left(f(\ue)-\frac{ \int _\Omega f(\ue)}{\int _ \Omega
\sqrt{4W(\ue)}}\sqrt{4W(\ue)}\right)\quad \text{ in }
\Omega\times(0,\infty),
\end{equation}
supplemented with the homogeneous Neumann boundary conditions
\begin{equation}\label{boundary}
 \frac{\partial \ue}{\partial \nu} (x,t) = 0 \quad\text{ on
} \partial \Omega\times(0,\infty),
\end{equation}
and the initial conditions
\begin{equation}\label{initial}
 \ue(x,0)=\gep (x) \quad\text{ in }\Omega.
 \end{equation}
Here $\Omega$ is a smooth bounded domain in $\R^N$ ($N\geq 2$) and
$\nu$ is the Euclidian unit normal vector exterior to $\partial
\om$. The small parameter $\ep>0$ is related to the thickness of a
diffuse interfacial layer. The term
\begin{equation}\label{etoile}
-\frac{ \int _\Omega f(\ue(x,t))\,dx}{\int _ \Omega
\sqrt{4W(\ue(x,t))}\,dx}\sqrt{4W(\ue(x,t))}
\end{equation}
 can be understood as a
Lagrange multiplier for the mass constraint
\begin{equation}\label{masse-conservee}
\frac d {dt} \int _\Omega \ue(x,t)\,dx=0.
\end{equation}
Let us notice that \eqref{etoile} combines nonlocal {\it and}
local effects (see below).

The nonlinearity is given by $f(u):=-W'(u)$, where $W(u)$ is a
double-well potential with equal well-depth, taking its global
minimum value at $u=\pm 1$. More precisely we assume that $f$ is
$C^2$ and has exactly three zeros $-1<0<+1$ such that
\begin{equation}\label{der-f}
f'(\pm 1)<0, \quad f'(0)>0\quad\ \hbox{(bistable nonlinearity)},
\end{equation}
and that $f$ is odd, which implies in turn that
\begin{equation}\label{int-f}
\int _ {-1} ^ {+1} f(u)\,du=0\quad\ \hbox{(balanced
nonlinearity)}.
\end{equation}
The condition \eqref{der-f} implies that the potential $W(u)$
attains its local minima at $u=\pm 1$, and \eqref{int-f} implies
that $W(-1)=W(+1)$. In other words, the two stable zeros of $f$,
namely $\pm 1$, have ``balanced" stability. For the sake of
clarity, in the sequel we restrict ourselves to the case where
\begin{equation}\label{cubique}
f(u)=u(1-u^2), \quad W(u)=\frac{1}{4} (1-u^2)^2.
\end{equation}
This will slightly simplify the presentation of the asymptotic
expansions in Section \ref{s:formal} and is enough to capture all
the features of the problem.

The initial data $\gep$ are {\it well-prepared} in the sense that
they already have sharp transition layers whose profile depends on
$\ep$. The precise assumptions on $\gep$ will appear in Theorem
\ref{th:results}. For the moment, it is enough to note that
\begin{equation}\label{gamma-zero}
\lim _{\ep\to 0} \gep=\begin{cases}
\, -1 &\text{ in the region enclosed by $\Gamma _0$}\\
\, +1 &\text{ in the region enclosed between $\partial \Omega$ and
$\Gamma _0$,}
\end{cases}
\end{equation}
where $\Gamma _0 \subset\subset \Omega$ is a given smooth bounded
hypersurface without boundary.

Our goal is to investigate the behavior of the solutions $\ue$ of
\eqref{mass-modified}, \eqref{boundary}, \eqref{initial}, as $\ep
\to 0$.

\medskip

\noindent{\bf Related works and comments.} It is long known that,
even for not well-prepared initial data, the sharp interface limit
of the Allen-Cahn equation $\partial _t \ue=\Delta \ue +\ep
^{-2}f(\ue)$ moves by its mean curvature. As long as the classical
motion by mean curvature exists, it was proved in \cite{Che} and
an optimal estimate of the thickness of the transition layers was
provided in \cite{A-Hil-Mat}. Let us also mention that, recently,
the first term of the actual profile of the layers was identified
\cite{A-Mat}. If the mean curvature flow develops singularities in
finite time, then a generalized motion can be defined via
level-set methods and viscosity solutions, \cite{Eva-Spr1} and
\cite{Che-Gig-Got}. In this framework, the convergence of the
Allen-Cahn equation to generalized motion by mean curvature was
proved by Evans, Soner and Souganidis \cite{Eva-Son-Sou} and a
convergence rate was obtained in \cite{A-Dro-Mat}.

The above results rely on the construction of efficient sub- and
super-solutions. Nevertheless, when
 comparison principle does not hold, a different method exists for
well-prepared initial data. It was used e.g. by Mottoni and
Schatzman \cite{Mot-Sch2} for the Allen-Cahn equation (without
using the comparison principle!); Alikakos, Bates and Chen
\cite{Ali-Bat-Che} for the convergence of the Cahn-Hilliard
equation
\begin{equation}\label{hilliard}
\partial _t \ue +\Delta \left(\ep \Delta \ue +\eun f(\ue)\right)=0,
\end{equation}
to the Hele-Shaw problem; Caginalp and Chen \cite{Cag-Che} for the
phase field system... The idea is to first construct solutions
$\uk$ of an approximate problem thanks to matched asymptotic
expansions. Next, using the lower bound of a linearized operator
around such constructed solutions, an estimate of the error $\|\ue
-\uk\|_{L^p}$ is obtained for some $p\geq 2$.

 Using these technics, Chen, Hilhorst
and Logak \cite{Che-Hil-Log} considered the Allen-Cahn equation
with conservation of the mass
\begin{equation}\label{mass-cons}
\partial _t \ue=\Delta \ue+\edeux \left(f(\ue)-\frac{1}{|\Omega|}\int_\Omega
f(\ue)\right),
\end{equation}
proposed by \cite{Rub-Ste} as a model for phase separation in
binary mixture. They proved its convergence to the volume
preserving mean curvature flow
\begin{equation}\label{MMC-vol-preserv}
V_n=-\kappa +\frac{1}{|\Gamma_t|}\int_{\Gamma_t}\kappa\,
dH^{n-1}\quad\text{ on } \Gamma _t .
\end{equation}
Here $V_n$ denotes the velocity of each point of $\Gamma_t$ in the
normal exterior direction and $\kappa$ the sum of the principal
curvatures, i.e. $N-1$ times the mean curvature. For related
results, we also refer the reader to the works \cite{Bro-Sto}
(radial case, energy estimates) and \cite{Hen-Hil-Mim} (case of a
system).

In a recent work, Brassel and Bretin \cite{Bra-Bre} proposed the
mass conserving Allen-Cahn equation \eqref{mass-modified} as an
approximation for mean curvature flow with conservation of the
volume \eqref{MMC-vol-preserv}. According to their formal approach
and numerical computations, it seems that \lq\lq
\eqref{mass-modified} has better volume preservation properties
than \eqref{mass-cons}". Let us notice that, as far as the local
Allen-Cahn equation is concerned, such an improvement of the
accuracy of phase field solutions, thanks to an adequate
perturbation term, was already performed in \cite{Gar-Sti} or in
\cite{Cag-Che-Eck}.

 In the present paper we  prove the convergence of
\eqref{mass-modified} to \eqref{MMC-vol-preserv}. Observe that in
\eqref{mass-cons} the conservation of the mass
\eqref{masse-conservee} is ensured by the Lagrange multiplier $
-\frac{1}{|\Omega|}\int_\Omega f(\ue) $ which is nonlocal, whereas
in the considered equation \eqref{mass-modified} the Lagrange
multiplier \eqref{etoile} combines nonlocal and local effects. On
the one hand, this will make the outer expansions completely
independent of the inner ones, and will cancel the $\ep$ order
terms of all expansions (see Section \ref{s:formal}). On the other
hand, this makes the proof of Theorem \ref{th:results} much more
delicate since further accurate estimates are needed (see
subsection \ref{ss:penible}). In other words, in the study
\cite{Che-Hil-Log} of \eqref{mass-cons}, it turns out that the
nonlocal Lagrange multipliers \lq\lq disappear" while estimating
the error estimate $\ue-\uk$. This will not happen in our context
and our key point will be the following. Roughly speaking, our
estimates of subsection \ref{ss:penible} will make appear an
integral of the error {\it on} the limit hypersurface which must
be compared with the $L^2$ norm of the error. If the former is
small compared with the latter then the Gronwall's lemma is
enough. If, as expected, the error concentrates so that the former
becomes large compared with the latter, then the situation is
favorable: a \lq\lq sign minus" intends at decreasing the $L^2$
norm of the error (see subsection \ref{ss:penible} and Remark
\ref{rem:concentrate} for details).

To conclude let us mention the work of Golovaty \cite{Gol}, where
a related equation with a nonlocal/local Lagrange multiplier is
considered. The convergence to a weak (via viscosity solutions)
volume preserving motion by mean curvature is proved via energy
estimates. As mentioned before, our method is different and allows
to capture a fine error estimate between the actual solutions and
the constructed approximate solutions.

\section{Statement of the results}\label{s:results}

\noindent{\bf The flow \eqref{MMC-vol-preserv}.}  Let us first
recall a few interesting features of the averaged mean curvature
flow \eqref{MMC-vol-preserv}. It is volume preserving, area
shrinking and every Euclidian sphere is en equilibrium. The local
in time well posedness in a classical framework is well understood
(see Lemma \ref{lem:mmc} for a statement which is sufficient for
our purpose). It is also known that local classical solutions with
convex initial data turn out to be global. Additionally, there
exist non-convex hypersurfaces (close to spheres) whose flow is
global. For more details on the averaged mean curvature flow
\eqref{MMC-vol-preserv}, we refer the reader to \cite{Gag},
\cite{Hui}, \cite{Esc-Sim} and the references therein.

\begin{lem}[Volume preserving mean curvature flow]\label{lem:mmc}
Let $\om _0\subset\subset \om$ be a subdomain such that $\Gamma
_0:=\partial \om _0$ is a smooth hypersurface without boundary.
Then there is $T^{max}\in(0,\infty]$ such that the averaged mean
curvature flow \eqref{MMC-vol-preserv}, starting from $\Gamma _0$,
has a unique smooth solution $\cup_{0\leq  t <T^{max}}
(\Gamma_t\times \{t\})$ such that $\Gamma _t\subset\subset
\Omega$, for all $t\in [0,T^{max})$.
\end{lem}

In the sequel, for $\Gamma _0$ as in \eqref{gamma-zero}, we fix
$0<T<T^{max}$ and work on $[0,T]$. We define
$$
\Gamma:=\cup_{0\leq t \leq T} (\Gamma_t\times \{t\}), $$ and
denote by $\om _t$ the region enclosed by $\Gamma_t$. Let us
define the step function $\tilde u=\tilde u(x,t)$ by
\begin{equation}\label{u-tilde}
\tilde u(x,t):=\begin{cases}
\, -1 &\text{in } \om _t\\
\, +1 &\text{in } \om \setminus \overline {\om _t}
\end{cases} \quad\text{ for all } t\in[0,T],
\end{equation}
which represents the sharp interface limit of $\ue$ as $\ep \to
0$. Let $d$ be the signed distance function to $\Gamma$ defined by
\begin{equation}
d (x,t)=
\begin{cases}
-&\hspace{-10pt}\mbox{dist}(x,\Gamma _t)\quad\text{for }x\in\om_t \\
&\hspace{-10pt} \mbox{dist}(x,\Gamma _t) \quad \text{for } x\in\om
\setminus \overline{\om _t}.
\end{cases}
\end{equation}

\medskip

\noindent{\bf Main results.} We  rewrite equation
\eqref{mass-modified} as
\begin{equation}\label{mass-modifiedbis}
\partial_t \ue-\Delta \ue
-\edeux\left(f(\ue)-\ep\lame (t)\sqrt{4W(\ue)}\right)=0\quad
\text{ in }\Omega \times(0,\infty),
\end{equation}
by defining
\begin{equation}\label{def-lambda}
\ep\lame(t):= \frac{\int_\Omega f(\ue)}{\int _ \Omega
\sqrt{4W(\ue)}}=\frac{\int_\Omega \ue -{\ue} ^3}{\int _ \Omega
1-{\ue} ^2}.
\end{equation}
Our first main result consists in constructing accurate
approximate solutions.

\begin{thm}[Approximate solutions]\label{th:approx} Let us fix
an arbitrary integer $k>\max (N,4)$. Then there exist $(\uk(x,t),
\lamk(t))_{x\in \bar \Omega,\, 0\leq t\leq T}$ such that
\begin{equation}\label{uek}
 \partial _t \uk-\Delta \uk-\edeux
\left(f(\uk)-\ep\lamk(t)\sqrt{4W(\uk)}\right)=\delta
_{\ep,k}\quad\text{ in }\om \times(0,T),
 \end{equation}
 with
 \begin{equation}\label{reste}
 \Vert \delta
 _{\ep,k}\Vert _{L^\infty(\Omega\times(0,T))}=\mathcal O(\ep ^k)\quad\text{ as } \ep \to 0,%\quad\text{ and }
 % \sup _{0\leq t\leq T} \left|\int _\Omega \delta_{\ep,k}(x,t)\,dx\right|=\mathcal O(\ep ^{k+\frac
  %12})
 \end{equation}
 and
 \begin{equation}\label{uek-neumann}
 \frac{\partial \uk}{\partial \nu}(x,t) =0 \quad \text{ on }\partial \om
 \times(0,T),
 \end{equation}
\begin{equation}\label{uek-mass-cons}
\frac d{dt} \int_\om \uk(x,t)\,dx=0 \quad \text{for all }t\in
(0,T).
\end{equation}
\end{thm}

Observe that by integrating \eqref{uek} over $\Omega$ and using
\eqref{uek-neumann} and \eqref{uek-mass-cons}, we see that
\begin{equation}\label{presque-lambda}
\ep \lamk(t)=\frac{\int _\om f(\uk) + \mathcal O(\ep^{k+ 2})}{\int
_\om \sqrt{4W(\uk)}}.
\end{equation}

Then we prove the following estimate, in the $L^2$ norm, on the
error between the approximate solutions $\uk$ and the solutions
$\ue$.

\begin{thm}[Error estimate]\label{th:results} Let us fix an arbitrary integer $k>\max
(N,4)$. Let $\ue$ be the solution of \eqref{mass-modified},
\eqref{boundary}, \eqref{initial} with the initial conditions
satisfying \eqref{gamma-zero} and
\begin{equation}\label{assumptions-initial}
\gep(x)=\uk(x,0)+\phi_\ep(x)\in[-1,1],\quad \int _\Omega \phi _\ep
=0,\quad \Vert \phi _\ep \Vert _{L^2(\Omega)}=\mathcal O(\ep
^{k-\frac 12}).
\end{equation}
Then, there is $C>0$ such that, for $\ep >0$ small enough,
$$
\sup _{0\leq t \leq T} \Vert \ue(\cdot,t)-\uk(\cdot,t)\Vert
_{L^2(\Omega)}\leq C \ep ^{k-\frac 12}.
$$
\end{thm}

Let us notice that, since $-1\leq \gep \leq 1$, it follows from
the maximum principle that $-1\leq \ue \leq 1$. Also since $\gep
\not \equiv 1$ and $\gep \not \equiv -1$, the conservation of the
mass implies $\ue\not \equiv 1$ and $\ue \not\equiv -1$, which
shows that the definition of $\ep \lame (t)$ in \eqref{def-lambda}
actually makes sense.

As it will be clear from our construction in Section
\ref{s:approx}, the approximate solutions satisfy
$$
\Vert \uk -\tilde u\Vert _{L^\infty(\{(x,t):\, |d(x,t)|\geq \sqrt
\ep\})}=\mathcal O(\ep^{k+2}),\quad \text{ as } \ep \to 0,
$$
with $\tilde u$  the sharp interface limit defined in
\eqref{u-tilde} via the volume preserving mean curvature flow
\eqref{MMC-vol-preserv} starting from $\Gamma _0$. We can
therefore interpret Theorem \ref{th:results} as a result of
convergence of the mass conserving Allen-Cahn equation
\eqref{mass-modified} to the volume preserving mean curvature flow
\eqref{MMC-vol-preserv}:
$$
\sup _{0\leq t \leq T} \Vert \ue(\cdot,t)-\tilde u(\cdot,t)\Vert
_{L^2(\Omega)}=\mathcal O (\ep ^{1/4}),\quad \text{ as } \ep \to
0.
$$

\medskip

\noindent{\bf Organization of the paper.} The organization of this
paper is as follows. In Section \ref{s:prelim} we present the
needed tools which are by now rather classical. In Section
\ref{s:formal}, we perform formal asymptotic expansions of the
solutions $(\ue(x,t),\lame(t))$. This will enable to construct the
approximate solutions $(\uk(x,t),\lamk(t))$, and so to prove
Theorem \ref{th:approx}, in Section \ref{s:approx}. Last we prove
the error estimate of Theorem \ref{th:results} in Section
\ref{s:error}. In particular and as mentioned before, a precise
understanding of the error between the actual and the approximate
Lagrange multipliers will be necessary (see subsection
\ref{ss:penible}).

\begin{rem} Through the paper, the notation $\psi _\ep \approx
\sum _{i\geq 0} \ep ^i \psi _i$ represents asymptotic expansions
as $\ep \to 0$ and means that, for all integer $k$, $\psi _\ep
=\sum _{i=0}^k \ep ^i \psi _i +\mathcal O (\ep ^{k+1})$.
\end{rem}

\section{Preliminaries}\label{s:prelim}

For the present work to be self-contained, we recall here a few
properties which are classical in the works mentioned in the
introduction, \cite{Rub-Ste}, \cite{Ali-Bat-Che}, \cite{Mot-Sch2},
\cite{Cag-Che}, \cite{Cag-Che-Eck}, \cite{Che-Hil-Log},
\cite{Hen-Hil-Mim}, and the references therein.

\subsection{Some related linearized operators}\label{ss:operator}

We denote by $\theta _0 (\rho):=\tanh ( \frac{\rho}{\sqrt 2})$ the
standing wave solution of
\begin{equation*}
\left\{\begin{array}{ll} {\theta_0}''+f(\theta _0)=0\quad\text{ on
}\R,\vspace{3pt}\\
\theta_0(-\infty)=-1,\quad \theta _0 (0)=0,\quad
\,\theta_0(\infty)=1,
\end{array} \right.
\end{equation*}
which we expect to describe the transition layers of solutions
$\ue$ observed in the stretched variable. Note that, for all $m\in
\N$,
\begin{equation}\label{est-theta}
D^m_\rho[\theta _0(\rho)-(\pm 1)]=\mathcal O(e^{-\sqrt 2|\rho|})
\quad \text{ as }\rho\to\pm\infty.
\end{equation}

We then consider the one-dimensional underlying linearized
operator around $\theta_0$, acting on functions depending on the
variable $\rho$ by
\begin{equation}\label{linoperator}
{\mathcal L}u:=-{u}_{\rho\rho}-f'(\theta_0(\rho))u.
\end{equation}

\begin{lem}[Solvability condition and decay at infinity]\label{lem:solvability}
Let $A(\rho,s,t)$ be a smooth and bounded function on $\R \times U
\times [0,T]$, with $U\subset \R^{N-1}$ a compact set. Then, for
given $(s,t)\in U\times[0,T]$, the problem
\begin{equation}\nonumber
\left\{\begin{array}{l} {\mathcal L} \psi:=-\psi
_{\rho\rho}-f'(\theta _0(\rho))\psi=A(\rho,s,t) \quad\text{ on }\R,\vspace{3pt}\\
\psi(0,s,t)=0,\quad \psi(\cdot,s,t) \in L^\infty (\R),
\end{array} \right.
\end{equation}
has a solution (which is then unique) if and only if
\begin{equation}\label{condition-solvability}
\int_\R A(\rho,s,t){\theta _0} '(\rho)\,d \rho=0.
\end{equation}

Under the condition \eqref{condition-solvability}, assume moreover
that there are real constants $A^\pm$ and an integer $i$ such
that, for all integers $m$, $n$, $l$,
\begin{equation}\label{decay-membre-droite}
D_\rho^m D_s^n D_t^l[A(\rho,s,t)-A^\pm]=\mathcal O(|\rho|^i
e^{-\sqrt 2|\rho|}) \quad \text{ as }\rho\to\pm\infty,
\end{equation}
uniformly in $(s,t)\in U\times[0,T]$. Then
\begin{equation}\label{decay}
D_\rho^m D_s^n D_t^l[\psi(\rho,s,t)-\frac{A^\pm}{f'(\pm
1)}]=\mathcal O(|\rho|^i e^{-\sqrt 2|\rho|})\quad \text{ as
}\rho\to\pm\infty,
\end{equation}
uniformly in $(s,t)\in U\times[0,T]$.
\end{lem}

\begin{proof} The lemma is rather standard (see
\cite{Ali-Bat-Che}, \cite{A-Hil-Mat} among others) and we only
give an outline of the proof. Multiplying the equation by
${\theta_0}'$ and integrating it by parts, we easily see that the
condition \eqref{condition-solvability} is necessary. Conversely,
suppose that this condition is satisfied. Then, since ${\theta _0}
'$ is a bounded positive solution to the homogeneous equation
$\psi_{\rho\rho} + f'(\theta _0(\rho))\psi=0$, one can use the
method of variation of constants to find the above solution $\psi$
explicitly:
$$
\psi(\rho,s,t)
=-{\theta_0}'(\rho)\int_0^\rho\left({{\theta_0}'}^{-2}(\zeta)
\int_\zeta^\infty A(\xi,s,t){\theta_0}'(\rho)\,d\xi\right)d\zeta.
$$
Using this expression along with the estimates
\eqref{decay-membre-droite} and \eqref{est-theta}, one then proves
\eqref{decay}.
\end{proof}

Note also, that after the construction of the approximate
solutions $\uk$, we shall need the estimate of the lower bound of
the spectrum of a perturbation of the self-adjoint operator
$-\Delta -\ep^{-2}f'(\uk)$ proved in \cite{Che3}. This will be
stated in Section \ref{s:error}.

\subsection{Geometrical preliminaries}\label{ss:geometrie} The
following geometrical preliminaries are borrowed from
\cite{Che-Hil-Log}, to which we refer for more details and proofs.

\medskip

\noindent {\bf Parametrization around $\Gamma$.}  As mentioned
before, we call $\Gamma=\cup_{0\leq t \leq T} (\Gamma_t\times
\{t\})$ the smooth solution of the volume preserving mean
curvature flow \eqref{MMC-vol-preserv}, starting from $\Gamma _0$;
we also denote by $\om _t$ the region enclosed by $\Gamma _t$. Let
$d$ be the signed distance function to $\Gamma$ defined by
\begin{equation}\label{eq:dist}
d (x,t)=
\begin{cases}
-&\hspace{-10pt}\mbox{dist}(x,\Gamma _t)\quad\text{for }x\in\om_t \\
&\hspace{-10pt} \mbox{dist}(x,\Gamma _t) \quad \text{for } x\in\om
\setminus \overline{\om _t}.
\end{cases}
\end{equation}
We remark that $d$ is smooth in a tubular neighborhood of
$\Gamma$, say in
$$
\mathcal
N_{3\delta}(\Gamma_t):=\{x\in\Omega:\,\,|d(x,t)|<3\delta\}\,,
$$
for some $\delta >0$. We choose a parametrization of $\Gamma _t$
by $X_0(s,t)$, with $s\in U \subset \R ^{N-1}$. We denote by
$n(s,t)$ the unit outer normal vector on $\partial \om _t=\Gamma
_t$. For any $0\leq t \leq T$, one can then define a
diffeomorphism from $(-3\delta,3\delta)\times U$ onto the tubular
neighborhood $\mathcal N_{3\delta}(\Gamma_t)$ by
$$
X(r,s,t)=X_0(s,t)+rn(s,t)=x \in \mathcal N_{3\delta}(\Gamma_t),
$$
whose inverse is denoted by $r=d(x,t)$,
$s=S(x,t):=(S^1(x,t),\cdots,S^{N-1}(x,t))$. Then $\nabla d$ is
constant along the normal lines to $\Gamma _t$, and the projection
$S(x,t)$ from $x$ on $\Gamma _t$ is given by
$X_0(S(x,t),t)=x-d(x,t)\nabla d(x,t)$. For $x=X_0(s,t)\in \Gamma
_t$ denote by $\kappa _i(s,t)$ the principal curvatures of $\Gamma
_t$ at point $x$ and by $V(s,t):=(X_0)_t(s,t).n(s,t)$ the normal
velocity of $\Gamma _t$ at point $x$. Then, one can see that
\begin{equation}\label{kappa}
\kappa (s,t):=\sum _{i=1}^{N-1}\kappa _i (s,t)=\Delta
d(X_0(s,t),t),
\end{equation}
\begin{equation}\label{b-un}
b_1(s,t):=-\sum _{i=1}^{N-1}{\kappa _i}^2 (s,t)=-(\nabla d. \nabla
\Delta d)(X_0(s,t),t),
\end{equation}
\begin{equation}
V(s,t):=(X_0)_t(s,t).n(s,t)=-d_t(X(r,s,t),t).
\end{equation}
In particular, $d_t(x,t)$ is independent of $r=d(x,t)$ in a small
enough tubular neighborhood of $\Gamma_t$. Changing coordinates
form $(x,t)$ to $(r,s,t)$, to any function $\phi(x,t)$ one can
associate the function $\tilde \phi (r,s,t)$ by
\begin{equation*}
\tilde \phi (r,s,t)=\phi(X_0(s,t)+rn(s,t),t)\, \quad \text{ or
}\quad \phi(x,t)=\tilde \phi (d(x,t),S(x,t),t).
\end{equation*}

\medskip

 \noindent {\bf The stretched variable.} In order to
describe the sharp transition layers of the solutions $\ue$ around
the limit interface, we now introduce a stretched variable. Let us
consider a graph over $\Gamma _t$ of the form
$$
\Gamma _t ^\ep=\{X(r,s,t):\, r=\ep h_\ep (s,t)\,,s \in U\},
$$
which is expected to represent the 0 level set, at time  $t$, of
the solutions  $\ue$. We define the stretched variable $\rho
(x,t)$ as \lq\lq the distance from $x$ to $\Gamma _t ^\ep$ in the
normal direction, divided by $\ep$", namely
\begin{equation}\label{correction}
\rho (x,t):=\frac{ d (x,t)-\ep h_\ep(S(x,t),t)}{\ep}.
\end{equation}
In the sequel, we use $(\rho,s,t)$ as independent variables for
the inner expansions. The link between the old and the new
variable is
$$
x=\hat
X(\rho,s,t):=X\left(\ep(\rho+h_\ep(s,t)),s,t\right)=X_0(s,t)+\ep(\rho+h_\ep
(s,t))n(s,t).
$$
Changing coordinates form $(x,t)$ to $(\rho,s,t)$, to any function
$\psi(x,t)$ one can associate the function $\hat \psi (\rho,s,t)$
by
\begin{equation}\label{changvar}
\hat \psi (\rho,s,t)=\psi(X_0(s,t)+\ep(\rho+h_\ep(s,t))n(s,t),t),
\quad
\end{equation}
or $\psi(x,t)=\hat \psi (\frac{d(x,t)-\ep
h_\ep(S(x,t),t)}{\ep},S(x,t),t)$. A computation then yields
\begin{equation}\label{dt-D}
\begin{aligned}
\ep^2(\partial_t \psi-\Delta\psi)=&-\hat\psi_{\rho\rho}-\ep(V+\Delta d)\hat\psi_\rho \\
&+\ep^2[\partial^\Gamma_t\hat\psi-\Delta^\Gamma \hat\psi -(\partial^\Gamma_th_\ep-\Delta^\Gamma h_\ep)\hat\psi_\rho]\\
&+\ep^2[2\nabla^\Gamma h_\ep. \nabla^\Gamma
\hat\psi_\rho-|\nabla^{\Gamma}h_\ep|^2\hat\psi_{\rho\rho} ].
\end{aligned}
\end{equation}
where $$
\partial _t ^\Gamma:=\partial _t+\sum_{i=1}^{N-1}
S_t^i\partial _{s^i},\;\nabla ^\Gamma :=\sum_{i=1}^{N-1} \nabla
S^i
\partial _{s^i}, \;\Delta ^\Gamma :=\sum_{i=1}^{N-1}\Delta S^i
\partial _{s^i}+\sum_{i,j=1}^{N-1}\nabla S^i.\nabla S^j \partial
_{s^is^j}.
$$
Here $\Delta d$ is evaluated at $(x,t)=(X_0(s,t)+\ep(\rho+
h_\ep(s,t))n(s,t),t)$, so that \eqref{kappa} and \eqref{b-un}
imply
\begin{equation}\label{expansion-laplacien}
\begin{aligned}
\Delta d &=\Delta d(X_0(s,t)+\ep(\rho+ h_\ep(s,t))n(s,t),t)\\
&\approx \kappa(s,t)-\ep(\rho+h_\ep(s,t))b_1(s,t)-\sum_{i\geq2}
\ep^i(\rho+h_\ep(s,t))^i b_i(s,t),
\end{aligned}
\end{equation}
where $b_i(s,t)$ ($i\geq 2$) are some given functions only
depending on $\Gamma _t$.

Last, define $$ \ep J^\ep (\rho,s,t):=\partial \hat X
(\rho,s,t)/\partial (\rho,s) $$ the Jacobian of the transformation
$\hat X$ so that, in particular, $dx=\ep
J^\ep(\rho,s,t)\,dsd\rho$. Then, for all $\rho\in\R $, $s\in U$
and $0\leq t\leq T$, we have
\begin{equation}\label{Jacobi}
J^{\ep}(\rho,s,t)=\prod_{i=1}^{N-1}
[1+\ep(\rho+h^{\ep}(s,t))\kappa_i(s,t)].
\end{equation}

\section{Formal asymptotic expansions}\label{s:formal}

In this section, we perform formal expansions for the solutions
$\ue(x,t)$ of \eqref{mass-modifiedbis}. We start by {\it outer
expansions} to represent the solutions \lq\lq far from the limit
interface", then make {\it inner expansions} to describe the sharp
transition layers. Last, expansions of the nonlocal term $\lame
(t)$ are performed. In the meanwhile we shall also discover the
expansions of the correction terms $h_\ep(s,t)$ defined in
\eqref{correction}.

\medskip

We assume that the solutions $\ue(x,t)$ are of the form
\begin{equation}\label{u-out}
\ue(x,t)\approx \ue^\pm(t):=\pm 1+\ep u^\pm_1(t)+\ep^2
u^\pm_2(t)+\cdots \quad\text{(outer expansions),}
\end{equation}
for $x\in \Omega _t$ (corresponding to $\ue ^-(t)$), $x\in \Omega
\setminus \Omega _t$ (corresponding to $\ue ^+(t)$), and away from
the interface $\Gamma _t$, say in the region where $|d(x,t)|\geq
\sqrt \ep$ as we expect the width of the transition layers to be
$\mathcal O (\ep)$. Near the interface $\Gamma _t$, i.e. in the
region where $|d(x,t)|\leq \sqrt \ep$, we assume that the function
$\hat \ue (\rho,s,t)$
--- associated with $\ue(x,t)$ via the change of variables
\eqref{changvar}--- is written as
\begin{equation}\label{u-in}
\hat \ue(\rho,s,t)\approx u_0(\rho,s,t)+\ep u_1(\rho,s,t)+\ep^2
u_2(\rho,s,t)+\cdots\quad\text{(inner expansions).}
\end{equation}
We also require the matching conditions between outer and inner
expansions, that is, for all $i\in\N$,
\begin{equation}\label{condmatch}
u_i(\pm \infty,s,t)=u^{\pm} _i(t)\quad\text{(matching
conditions)},
\end{equation}
for all $(s,t)\in U\times[0,T]$. As we expect the set $\rho=0$ to
be the 0 level set of the solutions (see subsection
\ref{ss:geometrie}) we impose, for all $i\in\N$,
\begin{equation}\label{condtran}
u_i(0,s,t)=0 \quad\text{(normalization conditions)},
\end{equation}
for all $(s,t)\in U\times[0,T]$.

As far as the nonlocal term $\lame(t)$ is concerned we assume the
expansions
\begin{equation}\label{lambda-exp}
\lame (t)\approx
\lambda_0(t)+\ep\lambda_1(t)+\ep^2\lambda_2(t)+\cdots\quad\text{
(nonlocal term)}.
\end{equation}
Last, the distance correcting term $h_\ep(s,t)$ is assumed to be
described by
\begin{equation}\label{h-in}
\ep h_\ep(s,t)\approx \ep h_1(s,t)+\ep^2
h_2(s,t)+\cdots\quad\text{(distance correction term)},
\end{equation}
for all $(s,t) \in U \times [0,T]$.

In the following, by the (complete) expansion at order 1 we mean
$$
\{d(x,t),\lambda_0(t),u_1(\rho,s,t),u_1^\pm(t)\}\quad\text{
(expansion at order 1)},
$$
and by the
(complete) expansion at order $i\geq 2$ we mean
\begin{equation}\label{complete}
\{h_{i-1}(s,t),\lambda _{i-1}(t),u_{i}(\rho,s,t),u_i
^\pm(t)\}\quad\text{ (expansion at order $i\geq 2$)}.
\end{equation}
Let us also recall that we have chosen
$$
f(u)=u(1-u^2), \quad W(u)=\frac{1}{4} (1-u^2)^2.
$$

\subsection{Outer expansions}\label{ss:outer} By plugging the
outer expansions \eqref{u-out} and the expansion
\eqref{lambda-exp} into the nonlocal partial differential equation
\eqref{mass-modifiedbis}, we get
\begin{equation}\label{mass-out}
\ep^2(u^{\pm}_{\ep})'(t)=u^{\pm}_{\ep}(t)-(u^{\pm}_{\ep}(t))^3-\ep\lame(t)(1-(u^{\pm}_{\ep}(t))^2).
\end{equation}
Since $u^\pm _\ep(t)\approx \sum _{i\geq 0}\ep ^i u^\pm _i(t)$,
where $u_0^\pm(t)=\pm1$, an elementary computation yields
$$
-\ep \lame(t)(1-(u^{\pm}_{\ep}(t))^2)\approx \sum _{i\geq 1}
\left(\sum _{p+q=i\,,q\neq 0} \lambda
_p(t)\sum_{k+l=q}u_k^\pm(t)u_l^\pm(t)\right)\ep ^{i+1},
$$
and
$$
(u^{\pm}_{\ep}(t))^3 \approx \sum _{i\geq 0} \left(\sum _{p+q=i}
u_p^\pm(t)\sum_{k+l=q}u_k^\pm(t)u_l^\pm(t)\right)\ep ^{i}.
$$
Hence, collecting the $\ep$ terms in \eqref{mass-out}, we discover
$0=u_1^\pm(t)-3u_1^\pm(t)(u_0 ^\pm(t))^2$ so that $u_1^\pm (t)
\equiv 0$. Next, an induction easily shows that
$$
u_i^\pm(t)\equiv 0 \quad\text{ for all } i\geq 1.
$$
Therefore the outer expansions are already completely known and
are trivial:
\begin{equation}\label{outer-trivial}
u_\ep ^\pm (t) \equiv \pm 1.
\end{equation}
In other words, thanks to the adequate form of the Lagrange
multiplier, the outer expansions are independent of the expansion
of the nonlocal term. This is in contrast with the equation
considered in \cite{Che-Hil-Log}.

\subsection{Inner expansions}\label{ss:inner}

It follows from \eqref{dt-D} that, in the new variables, equation
\eqref{mass-modifiedbis} is recast as
\begin{eqnarray}
&\hat \ue _{\rho\rho}+ \hat \ue -(\hat \ue)^3=\ep\lame(t)(1-(\hat \ue)^2)-\ep(V+\Delta d)\hat \ue _\rho\label{eqpe} \\
&+\ep^2[\partial^\Gamma_t\hat \ue-\Delta^\Gamma \hat \ue
-(\partial^\Gamma_th_\ep-\Delta^\Gamma h_\ep)\hat
\ue _\rho]\nonumber\\
&+\ep^2[2\nabla^\Gamma h_\ep. \nabla^\Gamma \hat \ue
_\rho-|\nabla^{\Gamma}h_\ep|^2\hat \ue _{\rho\rho} ].\nonumber
\end{eqnarray}

\medskip

\noindent {\bf The $\ep ^0$ terms.} By collecting the $\ep^0$
terms above and using the normalization and matching conditions
\eqref{condmatch}, \eqref{condtran} we discover that
$u_0(\rho,s,t)=\theta _0(\rho)$, with $\theta _0$ the standing
wave solution of
\begin{equation}\label{eq-theta-zero}
\left\{\begin{array}{ll} {\theta_0}''+f(\theta _0)=0\quad\text{ on
}\R,\vspace{3pt}\\
\theta_0(-\infty)=-1,\quad \theta _0 (0)=0,\quad
\,\theta_0(\infty)=1.
\end{array} \right.
\end{equation}
Formally, this solution represents the first approximation of the
profile of the transition layers around the interface observed in
the stretched coordinates. Note that since $f(u)=u-u^3$, one can
even compute $\theta_0(\rho)=\tanh(\frac{\rho}{\sqrt{2}})$.

\medskip

\noindent {\bf The $\ep ^1$ terms.} Next, since $\hat \ue
(\rho,s,t)\approx \sum _{i\geq 0}u_i(\rho,s,t)\ep^i$, where
$u_0(\rho,s,t)=\theta _0(\rho)$, an elementary computation yields
\begin{equation}\label{truc}
\ep \lame(t)\left(1-(\hat \ue)^2(\rho,s,t)\right)\approx -\sum
_{i\geq 0} \left(\sum _{p+q=i} \lambda _p(t)\beta _q
(\rho,s,t)\right)\ep ^{i+1},
\end{equation}
where
\[
 \beta_q(\rho,s,t)=\begin{cases}
 {\theta_0}^2(\rho)-1 &\text{if } q=0\vspace{3pt}\\
 \sum_{k+l=q}u_{k}(\rho,s,t)u_{l}(\rho,s,t) &\text{if }q\geq 1,
 \end{cases}
\]
and also
\begin{equation}\label{cube}
(\hat \ue)^3(\rho,s,t) \approx \sum _{i\geq 0} \left(\sum _{p+q=i}
u_p(\rho,s,t)\sum_{k+l=q}u_k(\rho,s,t)u_l(\rho,s,t)\right)\ep
^{i}.
\end{equation}

Hence, plugging the expansion \eqref{expansion-laplacien} of
$\Delta d$ into \eqref{eqpe} and collecting the $\ep$ terms, we
discover
\begin{equation}\label{eq-u1}
\begin{aligned}
{\mathcal L}u_1:=-u_{1\rho\rho}-f'(\theta
_0(\rho))u_1=(V+\kappa)(s,t){\theta_0}'(\rho)-(1-{\theta_0}^2(\rho))\lambda_0(t).
\end{aligned}
\end{equation}
For the above equation to be solvable (see Lemma
\ref{lem:solvability} for details) it is necessary that, for all
$(s,t) \in U\times [0,T]$,
$$
\int _{\R}\mathcal L u_1(\rho,s,t) {\theta _0}'(\rho)\,d\rho=0,
$$
which in turn yields
\begin{equation}\label{eq-mouvement}
V(s,t)=-\kappa (s,t)+\sigma \lambda_0 (t),\quad\quad
\sigma:=\frac{\int_{\R}(1-{\theta_0}^2){\theta_0}'}{\int_{\R}{{\theta_0}'}^2}.
\end{equation}
As seen in subsection \ref{ss:geometrie} the above equation can be
recast as
\begin{equation}\label{condsolv0}
d_t (x,t)=\Delta d (x,t)-\sigma \lambda_0(t) \quad \text{ for }\;
 x\in \Gamma_t.
\end{equation}
Now, in view of \eqref{eq-theta-zero}, we can write $0=\int
_{-\infty}^z({\theta _0}''+f(\theta _0)){\theta _0}'=\int
_{-\infty}^z({\theta _0}''-W'(\theta _0)){\theta _0}'$ and find
the relation $1-{\theta _0}^2=\sqrt 2 {\theta _0}'$, so that
$\sigma =\sqrt 2$. Plugging this and \eqref{eq-mouvement} into
\eqref{eq-u1} we see that $\mathcal L u_1=0$. Therefore, the
normalization $u_1(0,s,t)=0$ implies
\begin{equation}\label{u1-nul}
u_1(\rho,s,t)\equiv 0.
\end{equation}
Again this is in contrast with the equation considered in
\cite{Che-Hil-Log}.

\medskip

 \noindent {\bf The $\ep ^i$ terms ($i\geq 2$).} Now,
taking advantage of $u_0(\rho,s,t)=\theta _0(\rho)$ and of
$u_1(\rho,s,t)\equiv 0$ we identify, for $i\geq 2$, the $\ep ^i$
terms in all terms appearing in \eqref{eqpe}. In the sequel we
omit the arguments of most of the functions and, by convention,
the sum $\sum _{a}^b$ is null if $b<a$.

 Using \eqref{cube} we see
that the $\ep ^i$ term in $\hat \ue _{\rho\rho}+ \hat \ue -(\hat
\ue)^3$ is
\begin{equation}\label{term1}
-\mathcal L u_i-\theta _0 \sum
_{k=2}^{i-2}u_ku_{i-k}-\sum_{p=2}^{i-2}u_p\sum_{k+l=i-p}u_ku_l
\quad\quad\text{ (term 1)}.
\end{equation}

In view of \eqref{truc}, the $\ep ^i$ term in $\ep\lame(t)(1-(\hat
\ue)^2)$ is
\begin{equation}\label{term2}
\lambda_{i-1}(1-{\theta _0}^2)-\sum_{p+q=i-1,q\neq 0}\lambda _p
\sum_{k+l=q}u_ku_l\quad\quad\text{ (term 2)}.
\end{equation}

In order to deal with the term $-\ep(V+\Delta d)\hat \ue_\rho$, we
first note that \eqref{expansion-laplacien} and \eqref{h-in} yield
the following expansion of the Laplacian
\begin{equation}
\label{Laplacian} \Delta d \approx \kappa - \sum _{i \geq
1}\left(b_1h_i+\delta _i\right)\ep ^i,
\end{equation}
with
\begin{equation}\label{delta-i}
\delta _i=\delta_i(\rho,s,t)=\sum _{k=0}^{i} c_k(s,t)\rho^k
\end{equation}
a polynomial function in $\rho$ of degree lower than $i$, whose
coefficients $c_k(s,t)$ are themselves polynomial in
$(h_1,...,h_{i-1})$ which are part of the formal expansions at
lower orders, and in $(b_1,...b_i)$ which are given functions.
Among others, we have $\delta _1(\rho,s,t)=b_1(s,t)\rho$ and
$\delta _ 2(\rho,s,t)=b_2(s,t)(\rho+h_1(s,t))^2$. Combining ${\ue}
_\rho \approx {\theta _0}'+\ep^2 u_{2\rho}+\cdots$ and
\eqref{Laplacian}, we next discover that the $\ep ^i$ term in
$-\ep(V+\Delta d)\hat \ue _\rho$ is
\begin{eqnarray}\label{term3}
&b_1h_{i-1}{\theta _0}'+\delta_{i-1}{\theta _0}'-\left(V
 +\kappa\right)u_{(i-1)\rho}\\
 \nonumber &+\sum _{p=1}^{i-3}\left(b_1h_p+\delta
_p\right)u_{(i-1-p)\rho}\quad\quad\quad\quad\text{ (term 3)}.
\end{eqnarray}

We see that the $\ep ^i$ term in $\ep^2[\partial^\Gamma_t\hat \ue
-\Delta^\Gamma \hat \ue -(\partial^\Gamma_th_\ep-\Delta^\Gamma
h_\ep)\hat \ue _\rho]$ is given by
\begin{equation}\label{term4}
(\partial ^\Gamma _t-\Delta^\Gamma)u_{i-2}-(\partial ^\Gamma
_t-\Delta^\Gamma)h_{i-1}{\theta_0}'-\sum_{p=1}^{i-3}(\partial
^\Gamma _t-\Delta^\Gamma)h_p u_{(i-1-p)\rho}\quad\quad\text{ (term
4)}.
\end{equation}

Note that
$$
|\ep\nabla ^\Gamma  h_\ep|^2 \approx \ep ^2|\nabla ^\Gamma
h_1|^2+\sum_{i\geq 3}\left(2\nabla ^\Gamma h_1 .\nabla ^\Gamma
h_{i-1}+\eta _i\right)\ep ^i,
$$
where
$$
\eta_i=\eta _i(s,t):=\sum_{p+q=i-2,p\neq 0,q\neq 0} \nabla ^\Gamma
h_{p+1}(s,t).\nabla ^\Gamma h_{q+1}(s,t)
$$
depends only on the derivatives of $h_1$,...,$h_{i-2}$. Combining
this with $\hat \ue _{\rho\rho}\approx {\theta _0}''+\ep ^2
u_{2\rho\rho}+\cdots$, we discover that the $\ep^i$ term in $-\ep
^2|\nabla^{\Gamma}h_\ep|^2\hat \ue _{\rho\rho}$ is
\begin{equation}\label{term5}- \beta _i(\nabla
^\Gamma h_1.\nabla ^\Gamma h_{i-1}){\theta _0}''-|\nabla ^\Gamma
h_1|^2u_{(i-2)\rho\rho}-\sum _{k=0}^{i-3}\alpha _k
u_{k\rho\rho}\quad\quad\text{ (term 5)},
\end{equation}
where $\alpha_k=\alpha _k(s,t)$ depends only on the derivatives of
$h_1$,...,$h_{i-2}$ and $\beta _2=0$, $\beta _i=2$ if $i\geq 3$.

Last, since $\nabla ^\Gamma \hat \ue _\rho \approx \ep^2 \nabla ^
\Gamma u_{2\rho}+\cdots$, we see that the $\ep^i$ term in
$\ep^2[2\nabla^\Gamma h_\ep. \nabla^\Gamma \hat \ue _\rho ]$ is
\begin{equation}\label{term6}
2\sum _{k=2}^{i-2}\nabla ^\Gamma h_{i-1-k}.\nabla ^\Gamma
u_{k\rho}\quad\quad\text{ (term 6)}.
\end{equation}
%\begin{equation}\label{term6}
%\nabla^\Gamma h_1.\nabla ^\Gamma u_{(i-2)\rho}+\sum
%_{k=2}^{i-3}\nabla ^\Gamma h_{i-1-k}.\nabla ^\Gamma
%u_{k\rho}\quad\quad\text{ (term 6)}.
%\end{equation}

Hence, in view of the six terms appearing in \eqref{term1},
\eqref{term2}, \eqref{term3}, \eqref{term4}, \eqref{term5},
\eqref{term6}, when we collect the $\ep ^i$ term ($i\geq 2$) in
\eqref{eqpe} we face up to
\begin{equation}\label{ordre-i}
\mathcal L u_i=(\mathcal M ^\Gamma h_{i-1}){\theta _0}'-(1-{\theta
_0}^2)\lambda _{i-1}+\beta _i(\nabla ^\Gamma h_1.\nabla ^\Gamma
h_{i-1}){\theta _0}''+|\nabla ^\Gamma
h_1|^2u_{(i-2)\rho\rho}+R_{i-1}
\end{equation} where ${\mathcal M}^\Gamma$ denotes the linear
operator acting on functions $h(s,t)$ by
\begin{equation}\label{opM}
{\mathcal M}^\Gamma h:= \partial^{\Gamma}_t h-\Delta^\Gamma
h-b_1h,
\end{equation}
and where $R_{i-1}=R_{i-1}(\rho,s,t)$ contains all the remaining
terms. Since it is important that $R_{i-1}$ does not \lq\lq
contain" $h_{i-1}$, we have to leave $|\nabla ^\Gamma
h_1|^2u_{(i-2)\rho\rho}$ for the case $i=2$, but with a slight
abuse of notation we can \lq\lq insert" $|\nabla ^\Gamma
h_1|^2u_{(i-2)\rho\rho}$ in $R_{i-1}$ for $i\geq 3$. As an
example, for $i=2$ we see that
\begin{equation}\label{R2}
R_1(\rho,s,t)=-\delta_{1}(\rho,s,t) {\theta
_0}'(\rho)=-b_1(s,t)\rho {\theta _0}'(\rho),\end{equation} so that
we infer that, for all integers $m$, $n$, $l$,
\begin{equation}
D_\rho^m D_s^n D_t^l[R_1(\rho,s,t)]=\mathcal O(|\rho| e^{-\sqrt
2|\rho|}) \quad \text{ as }\rho\to\pm\infty,
\end{equation}
uniformly in $(s,t)$. Now, for $i\geq 3$, we isolate the \lq\lq
worst terms"
---which are the $\delta _i$'s--- in $R_{i-1}$ and write
\begin{equation}\label{R_i}
R_{i-1}=-\delta_{i-1}{\theta _0}'-\sum _{p=1}^{i-3}\delta
_pu_{(i-1-p)\rho}+r_{i-1},
\end{equation}
where $r_{i-1}=r_{i-1}(\rho,s,t)$ contains all the remaining
terms.

\begin{lem}[Decay of $R_{i-1}$]\label{lem:les-Ri} Let $i\geq 2$. Assume that, for any
$1\leq k\leq i-1$, there holds that, for all integers $m$, $n$,
$l$,
\begin{equation}\label{hyp}
D_\rho^m D_s^n D_t^l[u_k(\rho,s,t)]=\mathcal O(|\rho|^{k-1}
e^{-\sqrt 2|\rho|}) \quad \text{ as }\rho\to\pm\infty,
\end{equation}
uniformly in $(s,t)\in U\times[0,T]$. Then, for all integers $m$,
$n$, $l$
\begin{equation}\label{estimation-R-i}
D_\rho^m D_s^n D_t^l[R_{i-1}(\rho,s,t)]=\mathcal O(|\rho|^{i-1}
e^{-\sqrt 2|\rho|}) \quad \text{ as }\rho\to\pm\infty,
\end{equation}
uniformly in $(s,t)\in U\times[0,T]$.
\end{lem}

\begin{proof} Let us have a look at expression \eqref{R_i} of
$R_{i-1}$. By a tedious but straightforward examination we see
that $r_{i-1}(\rho,s,t)$ depends only on
\begin{itemize}
\item $V(s,t)$, $\kappa (s,t)$, $b_1(s,t),...,b_{i}(s,t)$ which
are bounded given functions \item $\lambda
_0(t),...,\lambda_{i-2}(t)$ \item $h_1(s,t),...,h_{i-2}(s,t)$ and
their derivatives w.r.t. $s$ and $t$ \item $u_0(\rho,s,t)=\theta
(\rho), u_1(\rho,s,t)=0,...,u_{i-1}(\rho,s,t)$ and their
derivatives w.r.t. $\rho$, $s$ and $t$
\end{itemize}
in such a way that it is $\mathcal O(|\rho|^{i-2}e^{-\sqrt 2
|\rho|})$ as $\rho\to\pm\infty$. Concerning the term
$$
-\delta_{i-1}( \rho,s,t){\theta _0}'(\rho)-\sum _{p=1}^{i-3}\delta
_p(\rho,s,t)u_{(i-1-p)\rho}, $$ the fact that it behaves like
\eqref{estimation-R-i} follows from  \eqref{hyp} and the fact that
$\delta _p(\rho,s,t)$ grows like $|\rho|^p$, as seen in
\eqref{delta-i}.
\end{proof}

Now, in virtue of Lemma \ref{lem:solvability}, the solvability
condition for equation \eqref{ordre-i} yields, for all $(s,t)$,
$$
(\mathcal M ^\Gamma h_{i-1})(s,t)\int _{\R} {{\theta
_0}'}^2-\lambda _{i-1}(t)\int _{\R}(1-{\theta _0}^2){\theta
_0}'+\int_{\R} R_{i-1}(\cdot,s,t){\theta _0}'=0.
$$
Note that the term $-\beta _i(\nabla ^\Gamma h_1.\nabla ^\Gamma
h_{i-1}){\theta _0}''$ does not appear above since $\int _{\R}
{\theta _0}''{\theta _0}'=0$. Note also that the term $-|\nabla
^\Gamma h_1|^2u_{(i-2)\rho\rho}$ does not appear for the same
reason if $i=2$, and because it can be \lq\lq hidden" in $R_{i-1}$
for $i\geq 3$ without altering the fact that $R_{i-1}$ does not
depend on $h_{i-1}$. The above equality can be recast as
\begin{equation}\label{eq-hi}
({\mathcal M}^\Gamma
h_{i-1})(s,t)=\sigma\lambda_{i-1}(t)-{\sigma^*}\int_\R
R_{i-1}(\rho,s,t) {\theta_0}'(\rho)\, d\rho,
\end{equation}
with $\sigma$ defined in \eqref{eq-mouvement} and
$\sigma^*:=\left(\int _\R {\theta _0}'^2\right)^{-1}$. Note that,
thanks to $1-{\theta _0}^2=\sqrt 2 {\theta _0}'$, we have $\sigma
=\sqrt 2$ (as seen before) and also $\sigma ^*=\frac 34 \sqrt 2$.

Let us have a look at $i=2$. From \eqref{R2} and the fact that
$\int _{\R}\rho {{\theta_0}'}^2(\rho)\,d\rho=0$ (odd function), we
see that \eqref{eq-hi} reduces to
\begin{equation}\label{eq-h1}
({\mathcal M}^\Gamma h_{1})(s,t)=\sigma\lambda_{1}(t).
\end{equation}
Assume that $h_1$ satisfies the above equation. Then since
$u_1\equiv 0$ trivially satisfies \eqref{hyp}, Lemma
\ref{lem:les-Ri} implies that  $R_1(\rho,s,t)$ together with  its
derivatives are $\mathcal O(|\rho|e^{-\sqrt 2 |\rho|})$ as $\rho
\to \pm \infty$. It follows from Lemma \ref{lem:solvability} that
\begin{equation}\label{equation-deux}
\mathcal L u_2=(\mathcal M ^\Gamma h_{1}){\theta _0}'-(1-{\theta
_0}^2)\lambda _{1}+|\nabla ^\Gamma h_1|^2{\theta _0}''+R_1,
\end{equation}
admits a unique solution $u_2(\rho,s,t)$ such that $u_2(0,s,t)=0$,
which additionally satisfies $D_\rho^m D_s^n
D_t^l[u_2(\rho,s,t)]=\mathcal O(|\rho|e^{-\sqrt 2 |\rho|})$.

Now, an induction argument straightforwardly concludes the
construction of the inner expansions.

\begin{lem}[Construction by induction]\label{lem:construction-inner}
Let $i\geq 2$. Assume that, for all $1\leq k\leq i-1$ the term
$u_k$ is constructed such that
\begin{equation}\label{decay-termes-avant}
D_\rho^m D_s^n D_t^l[u_k(\rho,s,t)]=\mathcal O(|\rho|^{k-1}
e^{-\sqrt 2|\rho|}) \quad \text{ as }\rho\to\pm\infty,
\end{equation}
uniformly in $(s,t)\in U\times[0,T]$. Assume moreover that
$h_{i-1}(s,t)$ satisfies the solvability condition \eqref{eq-hi}.
Then one can construct $u_i(\rho,s,t)$ solution of \eqref{ordre-i}
such that $u_i(0,s,t)=0$ and
\begin{equation}\label{decay-terme}
D_\rho^m D_s^n D_t^l[u_i(\rho,s,t)]=\mathcal O(|\rho|^{i-1}
e^{-\sqrt 2|\rho|}) \quad \text{ as }\rho\to\pm\infty,
\end{equation}
uniformly in $(s,t)\in U\times[0,T]$.
\end{lem}

\begin{rem} Note that the cancellation $u_1\equiv 0$ implies that the term $u_i(\rho,s,t)$
appearing in the expansion of the solutions of
\eqref{mass-modified} behaves like $\mathcal
O(|\rho|^{i-1}e^{-\sqrt 2 |\rho|})$, where the term
$u_i(\rho,s,t)$ appearing in the expansion of the solutions of
\eqref{mass-cons} behaves like $\mathcal O(|\rho|^{i}e^{-\sqrt 2
|\rho|})$ (see \cite{Che-Hil-Log}).
\end{rem}

\subsection{Expansions of the nonlocal term $\lame (t)$ and the distance correction term
$h_\ep(s,t)$}\label{ss:lambda}

 By following \cite[subsection
5.4]{Che-Hil-Log} with $\sqrt \ep$ playing the role of $\delta$,
we see that an asymptotic expansion of the conservation of the
mass \eqref{masse-conservee} yields
\begin{equation}
0=\frac d {dt} \int _\Omega \ue(x,t)\,dt\approx I_1+I_2+I_3,
\end{equation}
where $I_1=0$, since in our case $\ue ^\pm (t)\equiv \pm 1$, and
\begin{equation}\label{defI2}
I_2:=\int _{|\rho|<1 /\sqrt \ep}\partial _t ^{\Gamma}\hat \ue
(\rho,s,t) \,\ep J^\ep(\rho,s,t)\,d\rho\, ds,
\end{equation}
\begin{equation}\label{defI3}
I_3:=\int_{|\rho|< 1 /\sqrt \ep } (-V-\ep\partial^{\Gamma}_t
h_\ep)(s,t) \,\partial _{\rho} \hat\ue
(\rho,s,t)\,J^{\ep}(\rho,s,t)\,d\rho\,ds,
\end{equation}

Combining $\partial^\Gamma_t:=\partial
_t+\sum_{i=1}^{N-1}S^i_t\partial_{s^i}$ with $u_0(\rho,s,t)=\theta
_0(\rho)$ and $u_1(\rho,s,t)\equiv 0$, we see that
$$
\partial^\Gamma_t \hat \ue (\rho,s,t)\approx \sum _{i\geq 2} \ep ^i[\partial
_t+\sum_{k=1}^{N-1}S^k_t\partial_{s^k}] u_i(\rho,s,t).
$$
In view of the above inner expansions, this implies
$$
\partial^\Gamma_t \hat \ue (\rho,s,t)\approx \sum _{i\geq 2} \ep ^i \mathcal O \left (|\rho|^{i-1}e^{-\sqrt 2
|\rho|}\right),
$$
where $\mathcal O \left (|\rho|^{i-1}e^{-\sqrt 2 |\rho|}\right)$
depends only on expansions at orders $\leq i-1$. By plugging this
into \eqref{defI2}, we get
$$
I_2\approx \sum _{i\geq 3} \ep ^i \gamma_{i-2},
$$
where $\gamma _{i-2}=\gamma_{i-2}(t)$ depends only on expansions
at orders $\leq i-2$.

We now turn to the term $I_3$. We expand
$$
\left(-V-\ep\partial^{\Gamma}_t h_\ep\right)(s,t) \approx
d_t(X_0(s,t),t)-\sum_{i\geq 1}\ep^i \partial^{\Gamma}_t h_i (s,t),
$$
and
$$
\partial _{\rho} \hat\ue (\rho,s,t)\approx {\theta_0}'(\rho)+\sum_{i\geq2}\ep^i
\partial_{\rho}
u_i (\rho,s,t).
$$
Expanding the  Jacobian \eqref{Jacobi} and using \eqref{kappa}, we
get
$$
J^{\ep}(\rho,s,t)\approx 1+\Delta
d(X_0(s,t),t)\,\ep(\rho+h^{\ep}(s,t))+\sum_{i\geq2}\ep^i\mu_{i-1},
$$
where $\mu _{i-1}=\mu _{i-1}(\rho,s,t)$ depends only on expansions
at orders $\leq i-1$. Multiplying the three above equalities, we
see that the integrand in $I_3$ expands  as
$$
{\theta _0}' d_t+\ep{\theta _0}'\left[-\partial _t ^{\Gamma}
h_1+h_1 d_t\Delta d+\rho  d_t\Delta d\right]+\sum _{i\geq 2} \ep^i
{\theta _0}'(-\partial _t ^{\Gamma} h_i+h_i d_t\Delta d+\upsilon
_{i-1}),
$$
where $\upsilon _{i-1}=\upsilon _{i-1}(\rho,s,t)$ depends only on
expansions at orders $\leq i-1$. We integrate this over $s\in U$
and $|\rho|<  1/\sqrt \ep$ and, using $\int_{|\rho|<1/\sqrt \ep}
{\theta_0}'\approx \int_{\R} {\theta_0}'(\rho)\,d\rho=2$ and
$\int_{|\rho|<1/\sqrt \ep}\rho{\theta_0}'(\rho)\,d\rho=0$ (odd
function), we discover
\begin{eqnarray*}
\frac 12 I_3&\approx& \int _U d_t(s,t)\,ds+\ep \int _U(-\partial
_t ^{\Gamma}h_1+(d_t\Delta d)h_1)(s,t)\,ds\nonumber\\
&&+\sum_{i\geq 2} \ep ^i \left[\int _U(-\partial _t
^{\Gamma}h_i+(d_t\Delta d)h_i)(s,t)\,ds  +\omega_{i-1}\right],
\end{eqnarray*}
where $\omega _{i-1}=\omega _{i-1}(t)$ depends only on expansions
at orders $\leq i-1$. Using \eqref{condsolv0} to substitute $d_t$,
\eqref{eq-h1} to substitute $\partial _t ^\Gamma h_1$,
\eqref{eq-hi} to substitute $\partial _t ^\Gamma h_i$, we have
\begin{eqnarray*}
\frac 12 I_3&\approx& \int _U (\Delta d-\sigma \lambda_0)\,ds+\ep
\int _U(-\Delta^{\Gamma}
h_1-b_1h_1-\sigma\lambda_1+(d_t\Delta d)h_1)\,ds\nonumber\\
&&+\sum_{i\geq 2} \ep ^i \left[\int _U(-\Delta^{\Gamma}
h_i-b_1h_i-\sigma\lambda_i +(d_t\Delta d)h_i)\,ds +\zeta
_{i-1}\right],
\end{eqnarray*}
where $\zeta _{i-1}= \zeta_{i-1}(t)$ depends only on expansions at
orders $\leq i-1$.

Last, using $\int_U \Delta^{\Gamma}h_i\,ds=0$, we see that
$I_2+I_3 \approx 0$ reduces to
\begin{align}
\sigma\lambda_0(t)&= \overline{\Delta d(\cdot,t)}\label{eqlambda0}\\
\sigma\lambda_1(t)&= -\overline{[b_1(\cdot,t)-d_t(\cdot,t)\Delta
d(\cdot,t)]h_1(\cdot,t)}\label{eqlambda1}\\
\sigma\lambda_i(t)&= -\overline{[b_1(\cdot,t)-d_t(\cdot,t)\Delta
d(\cdot,t)]h_i(\cdot,t)}+\Lambda_{i-1}(t)\quad (i\geq
2),\label{eqlambdai}
\end{align}
where $\overline{\phi(\cdot)}:=\frac{1}{|U|}\int_U \phi$ denotes
the average of $\phi$ over $\Gamma_t$ (parametrized by $U$), and
$\Lambda_{i-1}(t)$ depends only on expansions at orders $\leq
i-1$. Moreover if we plug \eqref{eqlambda0}, \eqref{eqlambda1} and
\eqref{eqlambdai} into \eqref{condsolv0}, \eqref{eq-h1} and
\eqref{eq-hi}, we have the following closed system for $d$,
$h_1$,.., $h_i$ on $U\times[0,T]$:
\begin{align}
d_t&=\Delta d - \overline{\Delta d(\cdot,t)}\label{eqdistance}\\
\partial_t^{\Gamma} h_1&=\Delta^{\Gamma}h_1+b_1h_1-\overline{[b_1(\cdot,t)-d_t(\cdot,t)\Delta
d(\cdot,t)]h_1(\cdot,t)}\label{eqh1}\\
\partial_t^{\Gamma} h_i&=\Delta^{\Gamma}h_i+b_1h_i-\overline{[b_1(\cdot,t)-d_t(\cdot,t)\Delta
d(\cdot,t)]h_i(\cdot,t)}+\Lambda_{i-1}(t) \quad (i\geq
2).\label{eqhi}
\end{align}

\section{The approximate solutions $\uk$, $\lamk$}\label{s:approx}

In order to construct our desired approximate solutions and prove
Theorem \ref{th:approx}, let us first explain how the previous
section enables to determine, at any order, the outer expansion
\eqref{u-out}, the inner expansion \eqref{u-in}, the expansion of
the nonlocal term \eqref{lambda-exp}, and the expansion of the
distance correction term \eqref{h-in}.

First, as seen before, the outer expansion \eqref{u-out} is
already completely known since $u_i^\pm(t)\equiv 0$ for all $i\geq
1$.

Recall that $\Gamma = \cup_{0\leq  t \leq T} (\Gamma_t\times
\{t\})$ denotes the unique smooth evolution of the volume
preserving mean curvature flow \eqref{MMC-vol-preserv} starting
from $\Gamma _0\subset\subset \Omega$, to which we associate the
signed distance function $d(x,t)$. Hence, defining $\lambda _0(t)$
as in \eqref{eqlambda0} and $u_1(\rho,s,t)\equiv 0$ as in
\eqref{u1-nul}, we are equipped with the first order expansion
\begin{equation}\label{ordrezero}
\{d(x,t),\lambda _0(t),u_1(\rho,s,t)\equiv 0\}.
\end{equation}
Next, since $\Gamma _t$ is a smooth hypersurface without boundary,
there is a unique smooth solution $h_1(s,t)$ to the parabolic
equation \eqref{eqh1}. Assuming $h_1(s,0)=0$ for $s\in U$, we see
that $h_1(s,t)\equiv 0$, which combined with \eqref{eqlambda1}
yields $\lambda _1(t)\equiv 0$. Notice that these cancellations
are consistent with the observation of \cite{Bra-Bre} that \lq\lq
\eqref{mass-modified} has better volume preserving properties than
the traditional mass conserving Allen-Cahn equation
\eqref{mass-cons}". In Section \ref{s:formal}, we have defined
$u_2(\rho,s,t)$ as the solution of \eqref{equation-deux}, which
now reduces to $\mathcal L u_2=-b_1(s,t)\rho {\theta _0}'(\rho)$.
This completes the second order expansion, namely
\begin{equation}\label{ordreun}
\{h_1(s,t)\equiv 0,\lambda _1(t)\equiv 0,u_2(\rho,s,t)\}.
\end{equation}
Now, for $i \geq 2$, let us assume that expansions
$\{h_{k-1}(s,t),\lambda _{k-1}(t),u_k(\rho,s,t)\}$ are constructed
for all $2\leq k \leq i$. Therefore we can construct
$\Lambda_{i-1}(t)$ appearing in \eqref{eqhi}. Assuming
$h_i(s,0)=0$ for $s\in U$, there is a unique smooth solution
$h_{i}(s,t)$ to the parabolic equation \eqref{eqhi}. This enables
to construct $\lambda _i(t)$ via \eqref{eqlambdai}. Now,
$h_i(s,t)$ satisfies the solvability condition \eqref{eq-hi} at
rank $i$, so that Lemma \ref{lem:construction-inner} provides
$u_{i+1}(\rho,s,t)$, the solution of \eqref{ordre-i} at rank $i+1$
with $u_{i+1}(0,s,t)=0$. This completes the construction of the
$i+1$-th order expansion $\{h_i(s,t),\lambda
_i(t),u_{i+1}(\rho,s,t)\}$.

Note also that, from the above induction argument, we also deduce
the behavior \eqref{decay-terme} for all the $u_i(\rho,s,t)$'s.

\medskip

\noindent{\bf Proof of Theorem \ref{th:approx}.} We are now in the
position to construct the approximate solutions as stated in
Theorem \ref{th:approx}. Let us fix an  integer $k>\max (N,4)$. We
define
$$
\begin{aligned}
\rho_{\ep,k}(x,t)&:=\frac 1 \ep \left[d (x,t)-\sum_{i=1}^{k+2}\ep^i h_i(S(x,t),t)\right]=\frac {d_{\ep,k}(x,t)}\ep,\\
u^{in}_{\ep,k}(x,t)&:=\theta_0(\rho_{\ep,k}(x,t))+\sum_{i=1}^{k+3}\ep^i u_i(\rho_{\ep,k}(x,t),S(x,t),t),\\
u^{out}_{\ep,k}(x,t)&:=\tilde u (x,t),\\
\lamk(t)&:=\lambda _0(t)+\sum _{i=1}^{k+2} \ep ^i \lambda _i(t),
\end{aligned}
$$
where $\tilde u$ is the sharp interface limit defined in
\eqref{u-tilde}.  We introduce a smooth cut-off function
$\zeta(z)=\zeta _\ep (z)$ such that
$$
\begin{cases}
\zeta(z)=1 &\text {if}\;|z|\leq \sqrt \ep,\\
\zeta(z)=0 &\text {if}\;|z|\geq 2\sqrt \ep,\\
0\leq z\zeta'(z)\leq 4 &\text {if}\;\sqrt \ep\leq |z|\leq 2\sqrt
\ep.
\end{cases}
$$
For $x\in \bar \Omega$ and $0\leq t\leq T$, we define
$$
\uk ^* (x,t):=\zeta(d(x,t))u^{in}_{\ep,k}(x,t)+[1-\zeta(d(x,t))]
u^{out}_{\ep,k}(x,t).
$$
If $\ep >0$ is small enough then the signed distance $d(x,t)$ is
smooth in the tubular neighborhood $\mathcal N _{3\sqrt \ep}
(\Gamma)$, and so is $u^{in}_{\ep,k}(x,t)$. This shows that $\uk
^*$ is smooth.

Plugging $(\uk ^*(x,t),\lame(t))$ into the left hand side member
of \eqref{uek}, we find a error term $\delta_{\ep,k}^*(x,t)$ which
is such that
\begin{itemize}
\item $\delta_{\ep,k}^*(x,t)=0$ on $\{|d(x,t)|\geq 2\sqrt \ep\}$
since, then, $\uk ^*=\uk ^{out}=\pm 1$, \item $\Vert \delta
 _{\ep,k}^*\Vert _{L^\infty}=\mathcal O(\ep
 ^{k+2})$ on $\{|d(x,t)|\leq \sqrt \ep\}$ since, then,  $\uk ^*=\uk ^{in}$
and  the expansions of Section \ref{s:formal} were done on this
purpose, \item $\Vert \delta
 _{\ep,k}^*\Vert _{L^\infty}=\mathcal O(\ep
 ^{k^*})$, for any integer $k^*$, on $\{\sqrt \ep \leq |d(x,t)|\leq 2 \sqrt \ep\}$ since, then,  the decaying estimates
\eqref{est-theta} and \eqref{decay-terme} imply that $\uk ^*-\uk
^{out}=\uk ^* - \pm 1= \mathcal O (e^{-\frac{\sqrt 2}{2\sqrt
\ep}})$, valid also after any differentiation.
\end{itemize}
Hence $\Vert \delta _{\ep,k}^*\Vert _{L^\infty(\Omega
\times(0,T))} =\mathcal O (\ep^{k+2})$, which is even better than
\eqref{uek}. Also $\uk ^*$ clearly satisfies \eqref{uek-neumann}.

Now, to ensure the conservation of the mass of the approximate
solutions, we add a correcting term (which depends only on time)
and define
$$
\uk(x,t):= \uk ^*(x,t)+\frac 1 {|\Omega|} \int_{\Omega}( \uk ^*
(x,0)- \uk^*(x,t))\,dx,
$$
which then satisfies \eqref{uek-mass-cons}, and still
\eqref{uek-neumann}. Note also that subsection \ref{ss:lambda}
implies that the correcting term
$$
\int_{\Omega} (\uk ^* (x,0)- \uk ^*
(x,t))\,dx=-\int_{\Omega}\int_0^t
\partial_t  \uk ^* (x,\tau)\,d\tau dx
$$
is $\mathcal O (\ep ^{k+2})$ together with its time derivative.
Hence, when we plug $\uk={\uk ^*}+\mathcal O(\ep ^{k+2})$ into the
left hand side member of \eqref{uek}, we find a error term $\delta
_{\ep,k}$ whose $L^\infty$ norm is $\mathcal O(\ep^{k})$. \qed

\section{Error estimate}\label{s:error}

We shall here prove the error estimate, namely Theorem
\ref{th:results}. For ease of notation, we drop most of the
subscripts $\ep$ and write $u$, $\lambda$, $u_k$, $\lambda _k$,
$\delta _k$ for $\ue$, $\lame$, $\uk$, $\lamk$, $\delta_{\ep,k}$
respectively. By $\Vert\cdot \Vert$, $\Vert \cdot \Vert _{2+p}$ we
always mean $\Vert \cdot \Vert _{L^2(\Omega)}$, $\Vert \cdot \Vert
_{L^{2+p}(\Omega)}$ respectively. In the sequel, we denote by $C$
various positive constants which may change from places to places
and are independent on $\ep>0$.

\medskip

Let us define the error
$$
R(x,t):=u(x,t)-u_k(x,t).
$$
Clearly $\Vert R\Vert _{L^\infty} \leq 3$. It follows from the
mass conservation properties \eqref{masse-conservee},
\eqref{uek-mass-cons}, and the initial conditions
\eqref{assumptions-initial} that
\begin{equation}\label{bidule}
\int _\om R(x,t)\,dx=0 \quad \text{ for all } 0\leq t\leq T, \quad
\Vert R(\cdot,0)\Vert =\mathcal O(\ep ^{k-\frac 1 2}).
\end{equation}

 We
successively subtract the approximate equation \eqref{uek} from
equation \eqref{mass-modified}, multiply by $R$ and then integrate
over $\Omega$. This yields
\begin{equation}\label{energie}
\begin{aligned}
\frac{1}{2}\frac{d}{dt}\int _\om R^2=&-\int _\om |\nabla
R|^2+\edeux \int _\om
f'(u_k)R^2\\
& +\edeux \int _\om (f(u)-f(u_k)-f'(u_k)R)R -\int _\om \delta _k
R-\edeux \Lambda,
\end{aligned}
\end{equation}
where
\begin{equation}\label{Lambda}
\Lambda=\Lambda(t):= \int _\om[\ep \lambda(1-u^2)R-\ep \lambda
_k(1-{u_k}^2)R].
\end{equation}
Since $(f(u)-f(u_k)-f'(u_k)R)R=-3u_k R^3-R^4=\mathcal O(R^{2+p})$,
where $p:=\min(\frac 4 N,1)$, we have
$$
\left|\edeux \int _\om (f(u)-f(u_k)-f'(u_k)R)R\right|\leq \edeux C
\Vert R\Vert _{2+p} ^{2+p}\leq \edeux C_1\Vert R \Vert ^p \Vert
\nabla R\Vert ^2,
$$
where we have used the interpolation result \cite[Lemma
1]{Che-Hil-Log}. We also have $\left |\int _\om \delta _k
R\right|\leq \Vert \delta _k \Vert _\infty \Vert R\Vert=\mathcal
O(\ep ^{k})\Vert R \Vert$, so that
\begin{equation}\label{energie2}
\begin{aligned}
\Vert R \Vert \frac d{dt} \Vert R\Vert \leq & -\int _\om |\nabla
R|^2+\edeux \int _\om f'(u_k)R^2\\
&+ \edeux C_1 \Vert R \Vert ^p \Vert \nabla R\Vert ^2+\mathcal
O(\ep ^{k})\Vert R \Vert-\edeux \Lambda .
\end{aligned}
\end{equation}
We shall estimate $\Lambda$ in the following subsection. As
mentioned before, this term is the main difference with the case
of a strictly nonlocal Lagrange multiplier: its analogous for
equation \eqref{mass-cons} is $(\ep \lambda -\ep \lambda _k)\int
_\om R$ which vanishes, see \cite{Che-Hil-Log}.

Since $k>\max(N,4)$ we have $k-\frac 12 >\frac 4 p=\frac 4
{\min(\frac 4 N,1)}$, so that the second estimate in
\eqref{bidule} allows to define $t_\ep>0$ by
\begin{equation}\label{temps}
t_\ep:=\sup \left \{t>0, \forall\,  0\leq \tau \leq t, \Vert
R(\cdot,\tau)\Vert \leq (2C_1) ^{-1/p}\ep ^{4/p}\right\}.
\end{equation}
We need to prove that $t_\ep =T$ and that the estimate $\mathcal O
(\ep ^{4/p})$ is actually improved to $\mathcal O(\ep ^{k-\frac
12})$. In the sequel we work on the time interval $[0,t_\ep]$.

\subsection{Error estimates between the nonlocal/local Lagrange multipliers}\label{ss:penible}

It follows from \eqref{presque-lambda} that the term $\Lambda$
under consideration is recast as \begin{equation}\label{qqch}
\Lambda=\frac A B E- \frac{A_k}{B_k}E_k+\frac{\mathcal O(\ep
^{k+2})}{B_k}E_k, \end{equation}  where
$$
A_k=A_k(t):=\int _\om f(u_k), \; B_k=B_k(t):=\int _\om 1-{u_k}^2,
\; E_{k}=E_k(t):=\int _\om (1-{u_k}^2)R,
$$
and $A$, $B$, $E$ the same quantities with $u$ in place of $u_k$.

\begin{lem}[Some expansions]\label{lem:some} We have, as $\ep \to 0$,
$$
A_k=\ep ^2\alpha+\mathcal O(\ep ^3), \quad B_k=\ep \beta +\mathcal
O(\ep^2),
$$
where
$$
\alpha=\alpha(t):=\int _{U} \sum _{i=1}^{N-1} \kappa _i(s,t)\,ds
\int _\R \rho f(\theta _0(\rho))\,d\rho, \quad  \beta:=2\sqrt 2
|U|,
$$
and
$$
E_k=\mathcal O(\sqrt \ep \Vert R\Vert).
$$
\end{lem}

\begin{proof} We have seen in Section \ref{s:approx} that $u_k=u_k^*+\mathcal O(\ep ^{k+2})$ so it is enough
to deal with $A_k^*$, $B_k^*$ and $E_k^*$. The lemma is then
rather clear from the expansions of Section \ref{s:formal}. We
have
$$
\begin{aligned}
A_k^*&=\int _{|d(x,t)|\leq 2\sqrt \ep} f(u_k ^*)(x,t)\,dx=\int
_{|d(x,t)|\leq \sqrt \ep}f(u_k ^*)(x,t)\,dx+\mathcal O
(e^{-\frac{\sqrt 2}{2\sqrt
\ep}})\\
&=\int _U \int _{|\rho|\leq1/ \sqrt
\ep}f(\theta_0(\rho)+O(\ep^2))\ep J^\ep(\rho,s,t)
\,dsd\rho+\mathcal O (e^{-\frac{\sqrt 2}{2\sqrt \ep}}).
\end{aligned}
$$
Using $J^\ep(\rho,s,t)=1+\ep\rho \sum _{i=1}^{N-1}\kappa
_i(s,t)+\mathcal O(\ep ^2)$ and $\int _{|\rho|\leq 1/ \sqrt \ep}
f(\theta_0(\rho))\,d\rho=0$ (odd function), one obtains the
estimate for $A_k^*$. The estimate for $B_k^*$ follows the same
lines and is omitted. Last, the H\"older inequality yields
$|E_k|\leq (\int _\om (1-{u_k}^2)^2)^{1/2} \Vert R\Vert=\mathcal
O(\sqrt \ep\Vert R\Vert)$ since, again, $dx=\ep
J^\ep(\rho,s,t)\,ds\,d\rho$.
\end{proof}

As a first consequence of the above lemma, it follows from
\eqref{qqch} that
\begin{equation}\label{Lambda2}
\Lambda=\frac A B E- \frac{A_k}{B_k}E_k+\mathcal O(\ep ^{k+\frac 3
2})\Vert R\Vert.
\end{equation}
Next, in view of the above lemma, $u=u_k+R$ and $\Vert R \Vert
=\mathcal O (\ep ^{4/p})$, we can thus perform the following
expansions
$$
\begin{aligned}
A&=A_k+\int _\om (1-3{u_k}^2)R -3\int _\om u_k R^2 -\int _\om R^3\\
&=A_k+3 E_k-3\int _\om u_k R^2 +\mathcal O(\Vert R\Vert
_{2+p}^{2+p}),
\end{aligned}
$$
since $\int _\om R=0$,
$$
\begin{aligned}
B^{-1}&={B_k}^{-1}\left(1-\frac {2\int _\om u_k R}{B_k}-\frac{\int _\om R^2}{B_k}\right)^{-1}\\
&={B_k}^{-1}\left(1+\frac {2\int _\om u_k R}{B_k}+\frac{\int _\om
R^2}{B_k}+\left(\frac {2\int _\om u_k R}{B_k} \right)^2+\mathcal
O\left(\frac{\Vert R\Vert ^3}{\ep ^3}\right)\right),
\end{aligned}
$$
and
$$
\begin{aligned}
E=E_k-2\int _\om u_k R^2+\mathcal O(\Vert R\Vert _{2+p}^{2+p}).
\end{aligned}
$$
It follows that, using $E_k=\mathcal O(\sqrt \ep \Vert R\Vert)$
and $A_k=\mathcal O(\ep ^2)$ (see Lemma \ref{lem:some}),
$$
\begin{aligned}
AE=& A_kE_k-2A_k \int _\om u_k R^2+\mathcal O(\ep ^2 \Vert R\Vert _{2+p}^{2+p})+3{E_k}^2+\mathcal O(\sqrt \ep \Vert R \Vert ^3)\\
&+\mathcal O(\sqrt \ep
\Vert R\Vert\,\Vert R\Vert _{2+p}^{2+p})+\mathcal O(\sqrt \ep \Vert R \Vert ^3) +\mathcal O(\Vert R \Vert ^4)+\mathcal O(\Vert R \Vert^2\, \Vert R\Vert _{2+p}^{2+p})\\
&+\mathcal O (\Vert R\Vert _{2+p}^{2+p}\sqrt \ep \Vert R
\Vert)+\mathcal O (\Vert R\Vert _{2+p}^{2+p}\, \Vert R \Vert^2)
+\mathcal O(\Vert R\Vert _{2+p}^{4+2p})\\
=& A_k E_k +3 {E_k}^2-3E_k \int _\om u_k R^2 -2A_k \int _\om u_k
R^2+\mathcal O(\ep ^2\Vert R\Vert _{2+p}^{2+p})+\mathcal O(\Vert
R\Vert ^3),
\end{aligned}
$$
since $\Vert R \Vert _{2+p} ^{2+p}=\mathcal O(\Vert R \Vert ^2)$.
Then, using Lemma \ref{lem:some} and a methodical grading of the
$\mathcal O$'s (the typical used arguments being $\Vert R \Vert
_{2+p} ^{2+p}=\mathcal O(\Vert R \Vert ^2)$ and, e.g., $\Vert R
\Vert ^4=\mathcal O(\sqrt \ep \Vert R\Vert ^3)$ by the definition
of $t_\ep$), we arrive at
$$
\begin{aligned}
\frac A B E- \frac{A_k}{B_k}E_k=&{B_k}^{-1}\Big[3{E_k}^2-3E_k\int _\om u_k R^2+\frac{A_k}{B_k}E_k\int _\om 2 u_k R\\
&-2A_k \int _\om u_k R^2 +\mathcal O(\ep ^2\Vert R\Vert
_{2+p}^{2+p})+\mathcal O(\Vert R\Vert ^3)\Big],
\end{aligned}
$$
which in turn implies
\begin{equation*}
\begin{aligned}
\frac A B E- \frac{A_k}{B_k}E_k=&\frac{3{E_k}^2-3E_k\int _\om u_k R^2+\frac{A_k}{B_k}E_k\int _\om 2 u_k R}{B_k}\\
&-\frac{A_k}{B_k}\int _\om 2 u_k R^2 +\mathcal O(\ep \Vert R\Vert
_{2+p}^{2+p})+\mathcal O(\ep ^{-1}\Vert R\Vert ^3). \end{aligned}
\end{equation*}
Using Lemma \ref{lem:some} again, this implies
\begin{equation}\label{Lambda3}
\begin{aligned}
\frac A B E-
\frac{A_k}{B_k}E_k=&\frac{3{E_k}^2+\frac{A_k}{B_k}E_k\int _\om  2
u_k
R}{B_k}-\frac{A_k}{B_k}\int _\om 2 u_k R^2\\
& +\mathcal O(\ep \Vert R\Vert _{2+p}^{2+p})+\mathcal O(\ep
^{-1}\Vert R\Vert ^3).
\end{aligned}
\end{equation}
The term $-\frac{A_k}{B_k}\int _\om 2 u_k R^2$ is harmless since
it will be handled by the spectrum estimate Lemma
\ref{lem:spectre}. Let us analyze the fraction which is the worst
term. For $M>1$ to be selected later, define $\Vert R \Vert
_\mathcal T$, $\Vert R \Vert _{\mathcal T^c}$, the $L^2$ norms of
$R$ in the tube $\mathcal T:=\{(x,t):\, |d(x,t)|\leq M\ep\}$, the
complementary of the tube respectively:
$$
\Vert R \Vert _\mathcal T ^2:= \int _{\{|d(x,t)|\leq M\ep\}}
R^2(x,t)\,dx,\quad \Vert R \Vert _{\mathcal T^c} ^2:= \int
_{\{|d(x,t)|\geq M\ep\}} R^2(x,t)\,dx.
$$
Observe that the $\mathcal O(\ep)$ size of the tube allows to
write
$$
\left|\int _\mathcal T u_k R\right|\leq \left(\int _\mathcal T
{u_k}^2\right)^{1/2}\left(\int _\mathcal T R^2\right)^{1/2} \leq
C\sqrt \ep \Vert R\Vert _\mathcal T.
$$
Hence, using Lemma \ref{lem:some}, cutting $\int _\om=\int
_\mathcal T + \int _{\mathcal T ^c}$,  we get
$$
\begin{aligned}
\left| \frac{A_k}{B_k} E_k \int _\om 2u_k R \right|\leq & C \ep
|E_k|(\sqrt \ep \Vert R \Vert _\mathcal T +\Vert R \Vert _{\mathcal T^c})\\
\leq & C\ep \sqrt \ep |E_k|\,\Vert R\Vert _\mathcal T +C \ep
^{2/5} {E_k}^2+ C \ep^{8/5} \Vert R \Vert _{\mathcal T^c}^2.
\end{aligned}
$$
As a result
\begin{equation}\label{Lambda5}
\begin{aligned}
\frac{3{E_k}^2+\frac{A_k}{B_k}E_k\int _\om 2 u_k R}{B_k}\geq &
\frac{(3-C \ep ^{2/5}){E_k}^2-C\ep \sqrt \ep |E_k|\,\Vert R\Vert
_\mathcal T}{B_k}-C\ep ^{3/5}\Vert R\Vert
_{\mathcal T^c}^2\\
\geq & \frac{{E_k}^2-C\ep \sqrt \ep |E_k|\,\Vert R\Vert _\mathcal
T}{B_k}-C\ep ^{3/5}\Vert R\Vert _{\mathcal T^c}^2,
\end{aligned}
\end{equation}
for small $\ep >0$. Now, observe that
\begin{equation}\label{Lambda6}
{E_k}^2-C\ep \sqrt \ep |E_k|\,\Vert R\Vert _\mathcal T \geq
\begin{cases}
0 &\text{ if } |E_k|\geq  C\ep \sqrt \ep \Vert R\Vert _\mathcal T\\
 -C^2 \ep ^3 \Vert R\Vert _\mathcal T ^2 &\text{ if } |E_k|\leq  C\ep
\sqrt \ep \Vert R\Vert _\mathcal T.
\end{cases}
\end{equation}
\begin{rem}\label{rem:concentrate} The above inequality is the crucial one. One can
interpret it as follows. Following \cite[Proposition 2]{Bra-Bre},
we understand that $E_k$ behaves like the integral on the
hypersurface $\Gamma _t$:
$$
\ep \int _{d(x,t)=0} R (x,t)\,d\sigma.
$$
If $|E_k|=\left|\int _\om(1-{u_k}^2)R\right|$ is large w.r.t.
$\mathcal O (\ep \sqrt\ep \Vert R\Vert _\mathcal T)$ then
${E_k}^2- C\ep \sqrt \ep|E_k|\Vert R\Vert_\mathcal T \geq 0$,
which has the good sign to control the $L^2$ norm of $R$. In other
words, if the error \lq\lq intends" at concentrating on the
hypersurface, the situation is quite favorable. On the other hand,
if $|E_k|=\left|\int _\om (1-{u_k}^2)R\right|$ is small w.r.t.
$\mathcal O (\ep \sqrt \ep \Vert R\Vert _\mathcal T)$ then we get
the negative control $-\mathcal O(\ep ^2 \Vert R\Vert _\mathcal T
^2)$ (after dividing by $B_k$) which is enough for the Gronwall's
argument to work.
\end{rem}

Putting together \eqref{Lambda2}, \eqref{Lambda3},
\eqref{Lambda5}, \eqref{Lambda6} and $B_k=2\sqrt 2 |U| \ep
+\mathcal O(\ep ^2)$, we arrive at
\begin{equation}\label{LAMBDA}
\begin{aligned}
\Lambda \geq &-\frac{A_k}{B_k} \int _\om 2 u_k R^2 -C \ep ^{3/5}
\Vert R\Vert _{\mathcal T^c}^2-C\ep ^2 \Vert R\Vert _\mathcal
T^2\\&+\mathcal O(\ep ^{k+\frac 32})\Vert R\Vert +\mathcal O(\ep
\Vert R \Vert _{2+p}^{2+p})+\mathcal O(\ep ^{-1}\Vert R\Vert ^3).
\end{aligned}
\end{equation}

\subsection{Proof of Theorem \ref{th:results}}\label{ss:proof}

Equipped with the accurate estimate \eqref{LAMBDA}, we can now
conclude the proof of the error estimate by following the lines of
\cite{Che-Hil-Log}. Combining \eqref{energie2} with \eqref{LAMBDA}
and using the interpolation inequality $\Vert R\Vert _{2+p}
^{2+p}\leq C\Vert R \Vert ^p \Vert \nabla R\Vert ^2$, $\Vert R
\Vert _ \mathcal T\leq \Vert R \Vert$ and $\Vert R \Vert =\mathcal
O (\ep ^2)$ (thanks to the definition of $t_\ep$), we discover
\begin{equation*}
\begin{aligned}
\Vert R \Vert \frac d{dt} \Vert R\Vert \leq& -\int _\om |\nabla
R|^2+\edeux \int _\om
\left(f'(u_k)+\frac{A_k}{B_k} 2u_k\right)R^2\\
& +\edeux 2 C_1 \Vert R \Vert ^p \Vert \nabla R\Vert ^2+\edeux C
\ep ^{3/5}\Vert R\Vert _{\mathcal T^c} ^2\\
 &+C\Vert R\Vert ^2+
\mathcal O(\ep ^{k-\frac 12})\Vert R \Vert.\\
\end{aligned}
\end{equation*}
Since $ \ep^2\left(-\int _\om|\nabla R|^2+\frac 1{\ep ^2} \int
_\om (f'(u_k)+\frac{A_k}{B_k} 2u_k) R^2\right)\leq -\ep ^2 \Vert
\nabla R \Vert ^2+C\Vert R\Vert ^2$, we get
\begin{equation}\label{energie3}
\begin{aligned}
\Vert R \Vert \frac d{dt} \Vert R\Vert \leq & (1-\ep
^2)\left(-\int _\om|\nabla R|^2+\edeux \int _\om
\left(f'(u_k)+\frac{A_k}{B_k} 2u_k\right)
R^2\right)\\
& -\ep ^2 \Vert \nabla R \Vert ^2+\edeux 2 C_1 \Vert R \Vert
^p \Vert \nabla R\Vert ^2\\
& +\edeux C \ep ^{3/5}\Vert R\Vert _{\mathcal T^c} ^2+C\Vert
R\Vert ^2+
\mathcal O(\ep ^{k-\frac 12})\Vert R \Vert\\
\leq & (1-\ep ^2)\left(-\int _\om|\nabla R|^2+\edeux \int _\om
\left(f'(u_k)+\frac{A_k}{B_k} 2u_k\right)
R^2\right)\\
& +\edeux C \ep ^{3/5}\Vert R\Vert _{\mathcal T^c} ^2+C\Vert
R\Vert ^2 +\mathcal O(\ep ^{k-\frac 12})\Vert R \Vert,
\end{aligned}
\end{equation}
in view of the definition of $t_\ep$ in \eqref{temps}. In the
above inequality, let us write $\int _\om =\int _\mathcal T + \int
_{\mathcal T^c}$. In the complementary of the tube, observe that
$$
 \int
_{\mathcal T^c}\left(f'(u_k)+\frac{A_k}{B_k}2u_k+C\ep
^{3/5}\right)R^2= \int _{\{|d(x,t|\geq
M\ep\}}\left(f'(u_k)+\mathcal O(\ep ^{3/5})\right)R^2,
$$
is nonpositive if $M>0$ is large enough; this follows from the
form of the constructed $u_k$ in Section \ref{s:approx} ---
roughly speaking we have $u_k(x,t)=\theta
_0\left(\frac{d(x,t)+\mathcal O(\ep ^2)}\ep\right)+\mathcal O (\ep
^2)$--- $\theta _0(\pm \infty)=\pm 1$ and $f'(\pm 1)<0$. As a
result we collect
\begin{equation}\label{goal}
\begin{aligned}
\Vert R \Vert \frac d{dt} \Vert R\Vert \leq & (1-\ep
^2)\left(-\int _\mathcal T |\nabla R|^2+\edeux \int _\mathcal T
\left(f'(u_k)+\frac{A_k}{B_k} 2u_k\right)
R^2\right)\\
& +C\Vert R\Vert ^2 +\mathcal O(\ep ^{k-\frac 12})\Vert R \Vert.
\end{aligned}
\end{equation}
In some sense, the problem now reduces to a local estimate since
the linearized operator $-\Delta -\ep ^{-2} (f'(u_k)+\frac
{A_k}{B_k}2u_k)$ arises when studying the local unbalanced
Allen-Cahn equation
$$
\partial _t \ue=\Delta \ue
+\edeux\left(f(\ue)-\frac{A_k}{B_k}(1-{\ue}^2)\right),
$$
whose singular limit is \lq\lq mean curvature plus a forcing term"
(see, among others, \cite{A-Hil-Mat}). To conclude we need a
spectrum estimate of the unbalanced linearized operator around the
approximate solutions $u_k$, namely $-\Delta -\ep ^{-2}
(f'(u_k)+\frac {A_k}{B_k}2u_k)$. This directly follows from the
result of \cite{Che3} for the balanced case. For related results
on the spectrum of linearized operators for the Allen-Cahn
equation or the Cahn-Hilliard equation, we also refer to
\cite{Bat-Fif}, \cite{Ali-Fus, Ali-Fus2}, \cite{Mot-Sch2}.

\begin{lem}[Spectrum of the unbalanced linearized operator around
$u_k$ \cite{Che3}]\label{lem:spectre} There is $C^* >0$ such that
$$
-\int _ \mathcal T |\nabla R|^2+\edeux \int _ \mathcal T
\left(f'(u_k)+\frac{A_k}{B_k} 2u_k\right) R^2 \leq C^* \int
_\mathcal T R^2,
$$
for all $0<t\leq T$, all $0<\ep\leq 1$, all $R\in H^1(\Omega)$.
\end{lem}

\begin{proof} Observe that
$$
u_k(x,t)=\begin{cases} \theta_0\left(\frac{d_k(x,t)}\ep\right)+\mathcal O(\ep ^2) & \text{ if } |d(x,t)|\leq \sqrt \ep\\
\pm 1 +\mathcal O (\ep ^{k+1}) & \text{ if } |d(x,t)|\geq \sqrt
\ep.
\end{cases}
$$
Lemma \ref{lem:some} yields $\frac{A_k}{B_k}=\ep
\frac{\alpha(t)}{\beta}+\mathcal O(\ep ^2)$ so that we can write
$f'(u_k)+\frac{A_k}{B_k} 2u_k=f'(\overline {u_k})$, for some
$\overline{u_k}$ such that
$$
\overline{u_k}(x,t)=\begin{cases}
\theta_0\left(\frac{d_k(x,t)}\ep\right)-\ep
\frac{\alpha(t)}{3\beta} \theta _1\left(\frac{d_k(x,t)}\ep\right)
+\mathcal O(\ep ^2) & \text{ if } |d(x,t)|\leq \sqrt \ep\\
\pm 1 +\mathcal O (\ep ) & \text{ if } |d(x,t)|\geq \sqrt \ep,
\end{cases}
$$
where $\theta _1 \equiv 1$. In particular $\int _\R \theta
_1({\theta _0}')^2 f''(\theta _0)=\int _\R ({\theta _0}')^2
f''(\theta _0)=0$ (odd function) so that $\overline{u_k}$ has the
correct shape for \cite{Che3} to apply: see \cite[shape (3.8) and
proof of Theorem 5.1]{Ali-Bat-Che}, \cite[shape (16)]{Che-Hil-Log}
or \cite[Section 4]{Hen-Hil-Mim} for very related arguments.
Details are omitted.
\end{proof}

Combining the above lemma and \eqref{goal}, we end up with
$$
\frac d {dt} \Vert R \Vert \leq C \Vert R \Vert  +C \ep^{k-\frac
12}.
$$
The Gronwall's lemma then implies that, for all $0\leq t \leq
t_\ep$,
$$
\Vert R(\cdot,t)\Vert \leq (\Vert R(\cdot,0)\Vert+\ep ^{k-\frac
12})e^{Ct_\ep}=\mathcal O (\ep ^{k-\frac 12}),
$$
in view of \eqref{bidule}. Since $k-\frac 12 >\frac 4 p,$ this
shows that $t_\ep =T$ and that the estimate $\mathcal O (\ep
^{4/p})$ is actually improved to $\mathcal O(\ep ^{k-\frac 12})$.
 This completes the proof of Theorem
\ref{th:results}.\qed

\bigskip \noindent {\bf Acknowledgements.} The authors are grateful
to R\'emi Carles and Giorgio Fusco for valuable discussions on
this problem.


\begin{thebibliography}{ABCD}

%\bibitem{A-nonlocal} M. Alfaro, {\it Generation, motion and thickness of transition layers for a nonlocal Allen-Cahn equation},
% Nonlinear Anal. {\bf 72} (2010), no. 7-8, 3324--3336.

\bibitem{A-Dro-Mat} M.~Alfaro, J.~Droniou and H. Matano, {\it Convergence rate of the Allen-Cahn equation to generalized
motion by mean curvature},  J. Evol. Equ. {\bf 12} (2012),
267--294.

\bibitem{A-Hil-Mat} M. Alfaro, D. Hilhorst and H. Matano,
{\it The singular limit of the Allen-Cahn equation and the
FitzHugh-Nagumo system}, J. Differential Equations {\bf 245}
(2008), 505--565.

\bibitem{A-Mat} M. Alfaro and H. Matano,  {\it On the
validity of formal asymptotic expansions in Allen-Cahn equation
and FitzHugh-Nagumo system with generic initial data}, Discrete
Contin. Dyn. Syst. Ser. B. {\bf 17} (2012), 1639--1649.

\bibitem{Ali-Bat-Che} N. D. Alikakos, P. W. Bates et X.~Chen, {\it Convergence of the
Cahn-Hilliard equation to the Hele-Shaw model}, Arch. Rat. Mech.
Anal. {\bf 128} (1994), 165--205.

\bibitem{Ali-Fus} N. D. Alikakos and G. Fusco, {\it The spectrum
of the Cahn-Hilliard operator for generic interface in higher
space dimensions}, Indiana Univ. Math. J. {\bf 42} (1993),
637--674.

\bibitem{Ali-Fus2} N. D. Alikakos and G. Fusco, {\it Slow dynamics for the
Cahn-Hilliard equation in higher space dimensions. I. Spectral
estimates}, Comm. Partial Differential Equations {\bf 19} (1994),
1397--1447.


\bibitem{Bat-Fif}  P. W. Bates and P. C. Fife, {\it Spectral comparison principles for the Cahn-Hilliard and phase-field equations,
and time scales for coarsening},  Phys. D {\bf 43} (1990),
335–-348.

\bibitem{Bra-Bre}  M. Brassel and E. Bretin, {\it A modified phase field approximation for mean
curvature flow with conservation of the volume}, Math. Methods
Appl. Sci. {\bf 34} (2011), 1157--1180.


\bibitem{Bro-Sto} L.~Bronsard and B. Stoth, {\it Volume preserving mean curvature flow
as a limit of nonlocal Ginzburg-Landau equation}, SIAM. J. Math.
Anal. {\bf 28} (1997), 769--807.

\bibitem{Cag-Che} G. Caginalp and X. Chen, {\it Convergence of the phase field model
to its sharp interface limits}, European J. Appl. Math. {\bf 9}
(1998), 417--445.

\bibitem{Cag-Che-Eck} G. Caginalp, X. Chen and C. Eck, {\it Numerical tests of a
phase field model with second order accuracy}, SIAM J. Appl. Math.
{\bf 68} (2008), 1518–-1534.

\bibitem{Che} X.~Chen,
{\it Generation and propagation of interfaces for
reaction-diffusion equations}, J.~Differential Equations {\bf 96}
(1992), 116--141.

\bibitem{Che3} X.~Chen, {\it Spectrums for the Allen-Cahn, Cahn-Hilliard, and phase
field equations for generic interface}, Comm. Partial Differential
Equations  {\bf 19}  (1994), 1371--1395.


\bibitem{Che-Hil-Log} X.~Chen, D.~Hilhorst and E.~Logak,
{\it Mass conserving Allen-Cahn equation and volume preserving
mean curvature flow}, Interfaces Free Bound. {\bf 12} (2010),
527--549.

\bibitem{Che-Gig-Got} Y.~G.~Chen, Y.~Giga and S.~Goto, {\it Uniqueness and existence of viscosity solutions of
generalized mean curvature flow equations}, J.~Diff.~Geom. {\bf
33} (1991), 749--786.

\bibitem{Esc-Sim} J. Escher and G. Simonett, {\it The volume
preserving mean curvature flow near spheres},  Proc. Amer. Math.
Soc.  {\bf 126}  (1998), 2789--2796.

\bibitem{Eva-Son-Sou} L.~C.~Evans,
H.~M.~Soner and P.~E.~Souganidis, {\it Phase transitions and
generalized motion by mean curvature}, Comm.~Pure Appl.~Math. {\bf
45} (1992), 1097--1123.

\bibitem{Eva-Spr1} L.~C.~Evans and J. Spruck,
{\it Motion of level sets by mean curvature I}, J. Differential
Geometry {\bf 33} (1991), 635--681.

%\bibitem{Eva-Spr2} L.~C.~Evans and J. Spruck,
%{\it Motion of level sets by mean curvature II},
%Trans.~Amer.~Math.~Soc. {\bf 330} (1992), 321--332.

%\bibitem{Eva-Spr3} L.~C.~Evans and J. Spruck,
%{\it Motion of level sets by mean curvature III}, J. Geom. Anal.
%{\bf 2} (1992), 121--150.

\bibitem{Gag} M. Gage, {\it On an area-preserving evolution equation for plane
curves}, Nonlinear problems in geometry, Contemp. Math. 51, Amer.
Math. Soc., Providence, RI, 1986, 51--62.

\bibitem{Gar-Sti} H. Garcke and B. Stinner, {\it Second order phase field
asymptotics for multi-component systems}, Interfaces Free Bound.
{\bf 8} (2006), 131–-157.


\bibitem{Gol} D. Golovaty, {\it The volume-preserving motion by mean curvature as an asymptotic limit of reaction-diffusion
equations}, Quart. Appl. Math. {\bf 55} (1997), 243–--298.

\bibitem{Hen-Hil-Mim} M. Henry, D. Hilhorst and M. Mimura, {\it A
reaction-diffusion approximation to an area preserving mean
curvature flow coupled with a bulk equation}, Discrete Contin.
Dyn. Syst. Ser. S {\bf 4} (2011), 125--154.

\bibitem{Hui} G. Huisken, {\it The volume preserving mean
curvature flow}, J. Reine Angew. Math. {\bf 382} (1987), 35--48.

\bibitem{Mot-Sch2} P.~de Mottoni and M.~Schatzman,
{\it Geometrical evolution of developed interfaces},
Trans.~Amer.~Math.~Soc.~{\bf 347} (1995), 1533--1589.

\bibitem{Rub-Ste} J. Rubinstein and P. Sternberg, {\it
Nonlocal reaction-diffusion equations and nucleation}, IMA J.
Appl. Math. {\bf 48} (1992), 249–--264.

\end{thebibliography}
\end{document}